\documentclass{amsart}
\usepackage{amsfonts}
\usepackage{amsmath}
\usepackage{amssymb}
\usepackage{textcomp}
\usepackage{tikz}
\usepackage{enumerate}
\usetikzlibrary{matrix,arrows,decorations.pathmorphing}
\usetikzlibrary{decorations.markings}
\newtheorem{theorem}{Theorem}[section]

\newtheorem{lemma}[theorem]{Lemma}

\newtheorem{proposition}[theorem]{Proposition}

\newtheorem{corollary}[theorem]{Corollary}

\theoremstyle{definition}

\newtheorem{definition}[theorem]{Definition}

\newtheorem{example}[theorem]{Example}

\newtheorem{conjecture}[theorem]{Conjecture}

\newtheorem{remark}[theorem]{Remark}

\newtheorem *{Theorem A}{Theorem A}
\newtheorem *{Corollary B}{Corollary B}
\newtheorem *{Corollary C}{Corollary C}
\newtheorem *{Theorem B}{Theorem B}
\newtheorem *{Projective Schur's Lemma}{Projective Schur's Lemma}
\newtheorem *{Graded Artin Wedderburn Theorem}{Graded Artin Wedderburn Theorem}
\newtheorem *{Higgs' Conjecture}{Higgs' Conjecture}
\newtheorem *{Theorem C}{Theorem C}
\newtheorem *{Theorem D}{Theorem D}
\newtheorem *{Theorem D0}{Theorem D}
\newtheorem *{Theorem E}{Theorem E}
\newtheorem *{Theorem F}{Theorem F}
\newtheorem *{Theorem G}{Theorem G}
\newtheorem *{Theorem H}{Theorem H}
\newtheorem *{Corollary I}{Corollary I}
\newtheorem *{Theorem J}{Theorem J}
\newtheorem *{Corollary K}{Corollary K}
\newtheorem *{Theorem L}{Theorem L}
\newtheorem *{Question1}{Question 1}
\newtheorem *{Question2}{Question 2}
\newtheorem *{Problem 1}{Problem 1}
\newtheorem *{Problem 2}{Problem 2}
\newtheorem *{Problem 3}{Problem 3}

\newcommand{\C}{{\mathbb C}}

\newcommand{\N}{{\mathbb N}}

\numberwithin{equation}{section}

\begin{document}
\title[Maximal Connected Gradings and Intrinsic Fundamental Groups]{
Groups of central type, maximal Connected Gradings and Intrinsic Fundamental Groups of Complex Semisimple
Algebras}
\author{Yuval Ginosar}

\address{Department of Mathematics, University of Haifa, Haifa 3498838, Israel}
\email{ginosar@math.haifa.ac.il}
\author{Ofir Schnabel}
\thanks{The second author was supported by Minerva Stiftung}
\address{Institute of Algebra and Number Theory, Pfaffenwaldring 57, University of Stuttgart, Stuttgart 70569, Germany}
\email{os2519@yahoo.com}

\begin{abstract}
Maximal connected grading classes of a finite-dimensional algebra $A$
are in one-to-one correspondence with Galois covering classes of $A$ which admit no proper Galois covering
and therefore
are key in computing the intrinsic fundamental group
$\pi_1(A)$. Our first concern here is the algebras
$A=M_n(\mathbb{C})$. Their maximal connected
gradings turn out to be in one-to-one correspondence with the
Aut$(G)$-orbits of non-degenerate classes in $H^2(G,\C^*)$, where
$G$ runs over all groups of central type whose orders divide $n^2$.
We show that there exist groups of central type $G$ such that $H^2(G,\C^*)$ admits
more than one such orbit of non-degenerate classes.
We compute the family $\Lambda$ of positive integers $n$ such that there is a unique group of central
type of order $n^2$, namely $C_n\times C_n$. The family $\Lambda$ is of square-free integers and contains all
prime numbers. It is obtained by a full description of all groups of central type whose orders are cube-free.
We establish the maximal connected gradings of all
finite dimensional semisimple complex algebras using the fact that such gradings are determined by
dimensions of complex projective representations of finite groups.
In some cases we give a description of the corresponding fundamental groups.
\end{abstract}

\maketitle
\date{\today\vspace{-0.5cm}}
\tableofcontents

\bibliographystyle{abbrv}
\tikzset{node distance=3cm, auto}

\section{Introduction}\label{intro}\pagenumbering{arabic} \setcounter{page}{2}
A novel definition of the fundamental group $\pi_1(\mathcal{B})$ of a linear category $\mathcal{B}$ was suggested by
 C. Cibils, M.J. Redondo and A. Solotar \cite{cibilsintrinsic}. It is intrinsic in the sense that
``it does not depend on the presentation of the category by generators and relations".
Other notions of fundamental groups
are shown to be homomorphic images of this intrinsic group in case a universal covering exists.
In \cite{cibils2010},
the very same authors establish $\pi_1(A)$ for certain families of algebras $A$, considered
as 1-object linear categories,
essentially by computing the inverse limit of a diagram $\Delta(A)$
of surjective morphisms of the automorphism groups that are associated to the Galois coverings
of these categories. It is thus necessary to understand such coverings which admit no proper Galois covering,
alternately, the coverings whose corresponding grading classes are maximal with respect to the quotient partial ordering (see Definition \ref{scdiagdef}(2)).
In particular, a finite-dimensional algebra (as well as other linear categories) has a universal covering only if it admits a unique maximal connected
grading up to equivalence.
This enables Cibils, Redondo and Solotar to show that complex matrix algebras $M_n(\C)$
do not have universal coverings for $n>1$, by introducing two maximal connected matrix gradings of
totally different natures:
by the free group $\mathcal{F}_{n-1}$ and by the finite abelian group $C_n\times C_n$.

Recall that a {\it grading} of an algebra $A$ by a group $G$ is a vector space decomposition
\begin{equation}\label{eq:algebragrading}
A=\bigoplus _{g\in G} A_g
\end{equation}
such that $A_gA_h\subseteq A_{gh}$. The {\it support} of this grading is
$$\text{supp}_G(A)=\{g\in G|\text{dim}(A_g)\geq 1\}.$$
A $G$-grading of an algebra $A$ is {\it connected} if the support of this grading generates the group $G$.

The following example of a grading of $A=M_d(\mathbb{C})$
by the group $C_d\times C_d=\langle \sigma \rangle\times \langle \tau \rangle$ accompanies us throughout
this note.
\begin{example}\label{example}(see \cite[Proposition 4.4]{cibils2010},\cite[Example 0.4]{EK13}, \cite{HPP98})
Let $\xi _d$ be a $d$-{th} primitive root of unity.
Then the component of $A=M_d(\mathbb{C})$ that corresponds to $\sigma^i\tau^j\in C_d\times C_d$
is 1-dimensional and is given by $A_{\sigma^i\tau^j}=$ span$_{\C}\{B_{\sigma}^iB_{\tau}^j\}$, where
\begin{equation}\label{eq:twist}
\begin{array}{cc}
B_{\sigma}=
\left(
\begin{array}{ccccc}
0 & 0 & \cdots & 0 & 1 \\
1 & 0 & \cdots & 0 & 0 \\
\vdots & & \ddots &  &\vdots \\
0 & 0 & \cdots & 1 & 0
\end{array} \right),
& B_{\tau}=
\left(
\begin{array}{ccccc}
1  & 0 & 0 & \cdots & 0 \\
0 & \xi _d & 0 &\cdots & 0 \\
\vdots & &  & \ddots &\vdots \\
0 & 0 & 0 & \cdots  & \xi _d^{d-1}
\end{array} \right).
\end{array}
\end{equation}
\end{example}
Example \ref{example} is a prototype for a natural grading of a twisted group algebra
hereby explained.

Let $G$ be a group.
A {\it twisted group algebra} $\C^fG$ is a $G$-dimensional $\C$-vector space
endowed with multiplication, which is defined by a $G$-indexed basis $\{u_{g}\}_{g\in G}$
\begin{equation}\label{twistmult}
u_{g_1}u_{g_2}=f(g_1,g_2)u_{g_1g_2},\ \ g_1,g_2\in G,
\end{equation}
where $f\in Z^2(G,\C^*)$, that is a 2-cocycle.
A generalized Maschke's theorem says that if $G$ is finite then any twisted group algebra $\C^fG$ is semisimple. It affords a decomposition
\begin{equation}\label{AW}
\C^fG\cong\bigoplus_{[W]\in\text{Irr}(G,f)}\text{End}_{\C}(W),
\end{equation}
where Irr$(G,f)$ denotes the equivalence classes of irreducible complex $f$-projective representations of $G$ (see \cite[\S 3]{karpilovsky}).
Any twisted group algebra $\C^fG$ admits a natural $G$-grading by letting
$$(\C^fG)_g:=\text{span}_{\C}\{u_g\},\ \ \forall g\in G.$$
Such gradings are maximal connected by a similar proof to \cite[Proposition 5.1]{cibils2010}. Furthermore,
a theorem due to Y.A. Bahturin, S.K. Sehgal and M.V. Zaicev \cite{MR2488221} says that any group grading
on a complex semisimple algebra is obtained from twisted group algebras gradings by direct sums and inductions
(see Theorem \ref{th:BSZ}).

One can think of Example \ref{example} as a twisted group algebra $\C^f(C_d\times C_d)$, where the 2-cocycle $f\in Z^2(C_d\times C_d,\C^*)$ is obtained by the relations \eqref{twistmult}
when choosing basis elements
$0\neq u_g\in A_g$ for every $g\in C_d\times C_d$.
This example demonstrates that unlike group algebras,
there exist non-trivial groups $G$ and cocycles $f\in Z^2(G,\C^*)$ such that the twisted group algebra $\C^fG$ is itself a
matrix algebra, that is $|$Irr$(G,f)|=1$ in \eqref{AW}.
A group $G$ admitting such a phenomenon is called of {\it central
type}, and the 2-cocycle $f\in Z^2(G,\C^*)$ is called
{\it non-degenerate}.
By the definition it is clear that groups of central type are of square orders. Moreover, \eqref{AW} clearly yields
\begin{lemma}\label{dimprojrep}(see \cite[\S 1]{ShSh})
A cocycle $f\in Z^2(G,\C^*)$ is non-degenerate if and only if there exists $[W]\in\text{Irr}(G,f)$ with
$\dim_{\C}W=\sqrt{|G|}$.
\end{lemma}
Non-degeneracy is a cohomology class
property, hence we may refer to non-degenerate cohomology classes
$[f]\in H^2(G,\C^*)$. Furthermore, non-degeneracy is preserved
under a broader setting, that is under the following natural action of Aut$(G)$ on
$H^2(G,\C^*).$

Let $[f]\in H^2(G,\C^*)$ and let $\phi\in$Aut$(G)$.
Then
\begin{equation}\label{eq:action}
\begin{aligned}
\phi([f])&=[{\phi}(f)],\quad \text{ where}\\
{\phi}(f)(g_1,g_2):&=f({\phi}^{-1}(g_1),{\phi}^{-1}(g_2)).
\end{aligned}
\end{equation}
It is not hard to verify (see Proposition \ref{stable}) that the subset of
non-degenerate cohomology classes in $ H^2(G,\mathbb{C}^*)$ is
stable under this action.

Note that the $d$-th primitive root $\xi _d$ in Example
\ref{example} was not chosen explicitly. Different choices of
primitive roots yield non-cohomologous $2$-cocycles. However,
their corresponding cohomology classes lie in the same
Aut$(C_d\times C_d)$-orbit under the action \eqref{eq:action}
(see equation ~\eqref{eq:cdcd}).

This note is motivated by the study of the intrinsic fundamental groups of artinian semisimple complex algebras, regarded
as 1-object linear categories.
In this case, and more generally for every finite-dimensional algebra $A$, the set of equivalence classes of the connected gradings of $A$ (see Definition \ref{greqdef})
is partially ordered (see Proposition \ref{poset}).

We begin with algebras of complex matrices.
The grading classes of such algebras which are maximal with respect to this order are described by A. Elduque and
M. Kochetov \cite{EK13}, we formulate it in Theorem \ref{A} herein using a slightly different terminology (see \S\ref{terminology}).
This theorem, which is a consequence of
the above result of Bahturin, Sehgal and Zaicev, exhausts the maximal connected
gradings of these algebras up to equivalence in terms of orbits of non-degenerate cohomology classes under the action \eqref{eq:action}.
\begin{theorem}\label{A}(see \cite[Proposition 2.31 and Corollary 2.34]{EK13})\label{th:1-1}
There is a one-to-one correspondence between the maximal connected
grading classes of $M_n(\mathbb{C})$, and the set of pairs
$\mathcal{P}_n=\{(G,\gamma)\}$, where $G$ are groups of central
type of order dividing $n^2$, and $\gamma$ are Aut$(G)$-orbits of
non-degenerate cohomology classes in $H^2(G,\mathbb{C}^*)$.
\end{theorem}
If $G$ is a group of central type of order $d^2$ dividing $n^2$ and $\gamma$ is an orbit as above,
then the maximal connected class that corresponds to $(G,\gamma)\in\mathcal{P}_n$
is graded by the free product $\mathcal{F}_{\frac{n}{d}-1}*G$.
The correspondence in Theorem \ref{A} is described in \S\ref{gsa}.

By Theorem \ref{A}, every divisor $d$ of $n$ contributes at least one member to $\mathcal{P}_n$ by
a non-degenerate class of $C_d\times C_d$ as in Example \ref{example}.
That is, the number of non-equivalent maximal connected grading classes of $M_n(\mathbb{C})$
is bounded below by the number of divisors of $n$. In particular, if $n>1$ then $M_n(\mathbb{C})$ admits at least two graded-equivalence classes, and hence there is a
universal covering of $M_n(\mathbb{C})$ only if $n=1$ as observed
in \cite{cibils2010}.

Theorem \ref{A} poses two immediate problems. We record the first one as
\begin{Problem 1}
For a given $n$, classify the groups of central type of order $n^2$.
\end{Problem 1}
Problem 1 is fairly hard in its full generality.
Yet, we are able to give an answer to this problem for ``most" of the positive integers
(in terms of density, see Remark \ref{density}), namely those that are square-free. For any such number $n$, groups of central type of order $n^2$
have a nice description:

\begin{Theorem A}
Let $G$ be a group of order $n^2$, where $n$ is square-free.
Then $G$ is of central type if and only if $n$ decomposes as $n=m\cdot k$
(in particular $(m,k)=1$) such that $G^{\shortmid}\cong C_m\times C_m$, and
$$G\cong(C_m\times C_m)\rtimes (C_k\times C_k),$$
where the action of $C_k\times C_k$ on $C_m\times C_m$ is via SL$_2(m)$.
\end{Theorem A}
By Example \ref{example}, for every $n$ there is at least one group
of central type of order $n^2$, namely $C_n\times C_n$.
Theorem A, which is proven in \S\ref{Proof of Theorem A}, yields the following description of the family of positive integers
$$\Lambda:=\{n\in \N|\ \ C_n\times C_n\text{ is the \textit{only} group of central type of order }n^2\}.$$
\begin{Corollary B}
Let $n=\prod _{i=1}^rp_i^{k_i}$, where $\{p_i\}_{i=1}^r$ are distinct primes. Then $C_n\times C_n$ is the unique
isomorphism type of groups of central type of order $n^2$
if and only if $n$ is square-free ($k_i=1$) and
\begin{equation}\label{legsymb}
p_j\not \equiv \pm 1(\text{mod } p_i),\quad\forall 1\leq i,j\leq r.
\end{equation}
\end{Corollary B}
A proof of Corollary B can be found in \S\ref{Proof of Corollary B}.

As mentioned above, all the non-degenerate cohomology classes in
$H^2(C_d\times C_d,\C^*)$ lie in the same Aut$(C_d\times C_d)$-orbit. It is not hard to verify (see \S\ref{perct}) that the
subgroups of $C_n\times C_n$ of central type are exactly
$C_d\times C_d<C_n\times C_n$, for all the divisors $d$ of $n$.
Thus, if $n$ is in the family $\Lambda$ as described in Corollary B, then
$\mathcal{P}_n$ in Theorem \ref{A} consists of the unique Aut$(C_d\times
C_d)$-orbit of non-degenerate classes for every divisor $d$ of
$n$. Consequently, if $n=\prod _{i=1}^rp_i\in \Lambda$, then the cardinality of $\mathcal{P}_n$ is
equal to the number of divisors of $n$, which is $2^r$
($|\mathcal{P}_n|$ is strictly larger than the number of divisors
of $n$ if $n$ is not in $\Lambda$).

The second problem that arises from Theorem \ref{A} is recorded as
\begin{Problem 2}
Given a group $G$ of central type, classify the Aut$(G)$-orbits of
the non-degenerate cohomology classes in $H^2(G,\mathbb{C}^*)$.
\end{Problem 2}
In their paper \cite{aljadeff}, E. Aljadeff, D. Haile and M. Natapov present a family of groups of central
type, which contains all the abelian groups of central type. This family enjoys the property that for any
group $G$ in it, Aut$(G)$ acts \textit{transitively} on the
non-degenerate cohomology classes in $H^2(G,\mathbb{C}^*)$ \cite[Theorem 18]{aljadeff2}. One might conjecture that this transitivity property holds for
every group of central type. However, groups of central type with
a non-transitive action of Aut$(G)$ on their non-degenerate
cohomology classes do exist.
In \S\ref{non-transitive} we exhibit such groups.
Example \ref{example0} makes use of the structure of groups of central type of cube-free order given in Theorem A. It shows that such groups of central type may admit
non-degenerate classes in separate orbits under the corresponding automorphism action.
Nevertheless, such classes show similar behavior. More precisely, they belong to the same orbit under the more general action \eqref{genaction}.
In the next two examples, the non-degenerate cohomology classes are not in the same orbit under both actions and behave remarkably different.
Example \ref{example1} is of a family of groups of central type whose Sylow subgroups are all abelian.
Example \ref{example2} is a $2$-group of central type of order $256$.
It is worthwhile to point out that due to this non-transitivity of the automorphism action on the non-degenerate classes, a general
group-theoretic notion of symplectic groups should be assigned to
a non-degenerate form rather than to the group of central type admitting it (see Definition \ref{symplecticdef}).
With the concept of symplectic actions we obtain a closure property of groups of central type to certain extensions of
Hall subgroups and describe groups of central type with abelian $p$-Sylow subgroups.

One may aim at pushing the investigation of maximal connected grading classes further to general finite-dimensional semisimple complex algebras.
Any such algebra with Artin-Wedderburn decomposition
\begin{equation}\label{eq:kraoto3}
A=\bigoplus_nM_n(\C)^{b(n)}
\end{equation}
($b_n$ denotes the multiplicity of $M_n(\C)$ in this decomposition) can be recognized by the characteristic function
\begin{equation}\label{chiA}
\chi_A(z):=\sum_nb(n)n^{-z}.
\end{equation}
For twisted group algebras $\C^fG$, such functions encrypt dimensions of irreducible $f$-projective representations.
More precisely, to any $f\in Z^2(G,\C^*)$ yielding a decomposition
\eqref{AW}
one assigns the {\it $f$-projective zeta function}
\begin{equation}\label{projzeta}
\zeta_f(z):=\chi_{\C^fG}(z)=\sum_{[W]\in\text {Irr}(G,f)}(\dim(W))^{-z}.
\end{equation}
It is not hard to verify that cohomologous cocycles admit the same characteristic function \eqref{projzeta}.
Moreover, if $[f_1],[f_2]\in H^2(G,\C^*)$ are in the same Aut$(G)$-orbit $\gamma$, then $\zeta_{f_1}(z)=\zeta_{f_2}(z)$. The projective zeta function may therefore
be assigned also to an orbit, that is $\zeta_{\gamma}(z)$.

In order to classify the maximal connected grading class of semisimple algebras, define a formal term $\gamma^m$
for every Aut$(G)$-orbit $\gamma$ of cohomology classes in $H^2(G,\C^*)$ and every positive integer $m$. Regard $\gamma^m$ as {\it trivial} if $G$ is the one-element group and $m=1$.
A formal sum $\sum a(m,\gamma)\gamma^{m}$ with non-negative integer coefficients $a(m,\gamma)$ may contain orbits $\gamma$ in the cohomologies of several groups $G$.
\begin{Theorem C}
With the above notation, there is a one-to-one correspondence between the maximal connected
grading classes of a finite-dimensional complex semisimple algebra $A$, and the formal sums $\sum a(m,\gamma)\gamma^{m}$ with non-negative integer coefficients $a(m,\gamma)$
such that
$$
\chi_A(z)=\sum a(m,\gamma)m^{-z}\zeta_{\gamma}(z),$$ and such that $a(m,\gamma)\leq 1$ if $\gamma^m$ is trivial.
\end{Theorem C}
Theorem C is proven in \S\ref{proofD}.

The intrinsic fundamental group $\pi_1(\mathcal{B})$ is the inverse limit of a diagram $\Delta(\mathcal{B})$
of groups which grade the linear category $\mathcal{B}$
in a connected way (see \S\ref{fundint}).
In practice, in order to compute $\pi_1(A)$ of a finite-dimensional algebra $A$ it is enough to consider the maximal connected grading classes
of $A$ as well as their common quotients (Definition~\ref{def:quoalg}).
A forthcoming paper exploits the above study of maximal connecting grading classes
for estimating both $\pi _1 (M_n(\mathbb{C}))$, where $n$ is an integer satisfying the conditions in Corollary B, as well as $\pi _1 (M_p(\mathbb{C})\oplus M_p(\mathbb{C}))$
for primes $p$.

The intrinsic fundamental group of an artinian semisimple algebra may sometimes be isomorphic to the free product of
the intrinsic fundamental groups of its simple summands as shown in the following claim.
\begin{Theorem D}
Let $n_1,\cdots, n_s$ be distinct positive integers such that every prime $p$ is a factor of no more than two of them.
Then
$$\pi _1 (\bigoplus_{i=1}^k M_{n_i}(\mathbb{C}))=\coprod_{i=1}^k\pi _1 (M_{n_i}(\mathbb{C})).$$
\end{Theorem D}
Theorem D is proven in \S\ref{cgda}.
Special instances of this theorem are when the dimensions of the simple summands of a semisimple algebra are distinct and pairwise coprime.
Other instances of Theorem D are the bisimple algebras $M_{n_1}(\C)\oplus M_{n_2}(\C)$ with the extra condition $n_1\neq n_2$.
The maximal connected gradings of any bisimple algebra are described in \S\ref{bisimpsec}. This section also contains some results about groups of ``double central type",
which are main players in the description
of bisimple algebras.

{\bf Acknowledgements.}
We are grateful to F. Ced\'{o} and to M. Kochetov for their important comments, which helped us improve this note.
M. Kochetov also referred us to bibliography that we were not aware of.
We also thank D. Blanc and A. Solotar for valuable discussions.

\section{Group grading of vector spaces and algebras}
This section reviews notions and results about gradings that will be needed in the sequel.
Throughout this note, all vector spaces and algebras are finite-dimensional over $\C$, even though
these conditions are sometimes redundant.
The reader is referred to \cite{EK13} for an extensive account on grading for general spaces and algebras (not necessarily
associative).
The terminologies used in \cite{EK13} and herein do not always coincide, see \S\ref{terminology}.
\subsection{Graded vector spaces}\label{cps}
A grading of a vector space $V$ by a group $G$
is a decomposition of $V$ to subspaces indexed by the group elements
\begin{equation}\label{gradvecsp}
\mathcal{G}:V=\oplus _{g\in G} V_g.
\end{equation}
A subspace $U\subseteq V$  is said to be a {\it graded-subspace} if $$U=\oplus _{g\in G} (U\cap V_g).$$

The subspaces $V_g$ are the {\it homogeneous components} of $V$ and the vectors
$v\in V_g$ are called {\it homogeneous}.
A {\it graded basis} of $V$ is a basis consisting of homogeneous elements.

If \eqref{gradvecsp} is connected, i.e. its support generates $G$, and $U\subset V$ is a graded-subspace, then certainly its $G$-grading is not necessarily connected.
However, it is connectedly-graded by the subgroup
$$\langle g\in G |\ \  U\cap V_g\neq 0\rangle<G.$$ Focusing on connected gradings, $U$ is regarded as graded by this subgroup.

Let
\begin{equation}\label{2grsp}
\mathcal{G}_V:V=\oplus _{g\in G} V_g\quad , \quad \mathcal{G}_W:W=\oplus _{h\in H} W_h
\end{equation}
be two group-graded spaces. A {\it graded morphism} between these spaces is a pair $(\psi,\phi)$ of a linear transformation $\psi:V\to W$ and a group homomorphism
$\phi: G\to H$
such that $\psi(V_g)\subset W_{\phi(g)}$ for every $g\in G$.
The graded spaces \eqref{2grsp}
are {\it graded equivalent} if there is a graded morphism $(\psi,\phi)$ between them such that $\psi:V\to W$ and $\phi: G\to H$ are isomorphisms
of linear spaces and groups respectively.
These spaces
are {\it graded isomorphic} if, with the above notation, $G=H$ and
$\phi$ is the identity on $G$, that is
$\psi(V_g)=W_g$ for every $g\in G$.

The set $\mathcal{S}_G$ of $G$-graded isomorphism classes of finite-dimensional vector spaces
has a structure of a semiring with involution as follows.
Addition is by direct sum of representatives. Clearly, the zero-dimensional space plays the role of 0 in the
additive monoid $\mathcal{S}_G$.
Next, any two $G$-gradings \eqref{2grsp} of vector spaces
equip their tensor product with a natural convolution $G$-grading
$$\mathcal{G}_V\otimes\mathcal{G}_W:V\otimes W=\oplus_{g\in G}(V\otimes W)_g,$$ where
\begin{equation}\label{eq:tensvec}
(V\otimes W)_g=\text{span}_{\C}\{v\otimes w\in V_{g_1}\otimes W_{g_2} \mid g_1\cdot g_2=g\}.
\end{equation}
A $G$-grading \eqref{gradvecsp} determines a natural $G$-grading on its dual space $V^*=$Hom$(V,\C)$.
$$\mathcal{G}^*:V^*=\oplus _{g\in G} (V^*)_g,$$
 where
\begin{equation}\label{dual}
(V^*)_g:=\{\phi \in V^*\mid\phi|_{V_h}=0\quad \forall h\neq g^{-1}\}.
\end{equation}
The product \eqref{eq:tensvec} and the involution $\mathcal{G}\mapsto \mathcal{G}^*$
are well defined on graded equivalence classes as well as on graded isomorphism classes.
Clearly, the 1-dimensional space graded by the trivial element of $G$ serves as the unit in the semiring with involution
$\mathcal{S}_G$.

The semiring $\mathcal{S}_G$ is a left semimodule over itself in two different ways.
A class $[V]\in$$\mathcal{S}_G$ can act on $[W]\in$$\mathcal{S}_G$ either
by tensoring from the left
$$[W]\mapsto [V\otimes W],$$
or tensoring from the right by the dual
$$[W]\mapsto [W\otimes V^*].$$
We are interested in the combined action of these two. Define the corresponding {\it induced} $G$-graded space
\begin{equation}\label{indef}
\mathcal{G}_V\otimes\mathcal{G}_W\otimes\mathcal{G}_V^*:V\otimes W\otimes V^*=\oplus_{g\in G}(V\otimes W\otimes V^*)_g,
\end{equation}
where by \eqref{eq:tensvec}
$$(V\otimes W\otimes V^*)_g=
\text{span}_{\C}\{v\otimes w\otimes \phi\in V_{g_1}\otimes W_{g_2}\otimes (V^*)_{g_3} \mid g_1\cdot g_2\cdot g_3=g\}.$$
This determines the {\it left diagonal action} of $\mathcal{S}_G$ on itself
\begin{equation}\label{diagdef}
^{[V]}[W]:= [V\otimes W\otimes V^*].
\end{equation}

The semiring $\mathcal{S}_G$ can be identified with the group
semiring $\mathbb{N} G$ of the group $G$ over the nonnegative integers $\mathbb{N}$ as follows.
To any $G$-grading \eqref{gradvecsp} of a finite-dimensional vector space one can assign a character
$$\chi(\mathcal{G}):=\sum _{g\in G} \text{dim}_{\C}(V_g)g \in \mathbb{N}G.$$
The character satisfies
\begin{equation}\label{resprod}
\chi(\mathcal{G}_V\otimes\mathcal{G}_W)=\chi(\mathcal{G}_V)\cdot \chi(\mathcal{G}_W).
\end{equation}
The semiring also has a natural involution $\mathbb{N}G \to  \mathbb{N}G$ given by
$$(\sum _{g\in G} n_gg)^{\star} =\sum _{g\in G} n_gg^{-1}.$$
The following commutation property between the character and the two involutions is immediate.
\begin{equation}\label{invo}
\chi(\mathcal{G}^*)=\chi(\mathcal{G})^{\star}.
\end{equation}
Now, $\chi$ is a class function and can therefore be defined on the corresponding classes in $\mathcal{S}_G$.
Furthermore, it is not hard to check that $\chi:\mathcal{S}_G\to \mathbb{N} G$ is bijective.
By \eqref{resprod} and \eqref{invo} it is an isomorphism of semirings with involution.

When $G$ is finite, a special class $[\mathcal{G}_0]\in$$\mathcal{S}_G$ is singled out.
A $G$-graded vector space \eqref{gradvecsp} is in the class $[\mathcal{G}_0]$ if
dim$_{\C}(V_g)=1$ for every $g\in G$.
It is easily verified that the class $[\mathcal{G}_0]$ is a self-dual central element in $\mathcal{S}_G$,
whose character is given by
$$\chi([\mathcal{G}_0])=\sum_{g\in G}g\in \mathbb{N}G.$$
The class $[\mathcal{G}_0]$ is quasi-idempotent in the sense that its diagonal action on
elements of $\mathbb{N}G$ is the same as multiplying them by $|G|^2$.
Alternatively, the diagonal action of $[\mathcal{G}_0]$ on finite-dimensional $G$-graded
spaces is by duplicating them $|G|^2$ times.

The semiring $\N G$ is key in describing graded spaces and algebras.
The following equivalence relation $R_H$ on $\N G$
generalizes the well known partition of $G$ to $H$-cosets for $H\leq G$. It is recorded here for later use.
\begin{definition}\label{cosetequiv}
Let $G$ be a group and let $H \leq G$. The {\it right $H$-coset} of an element $x=\sum_{i=1}^r g_i\in \N G$
is the set of elements
$$\{\sum_{i=1}^r g_ih_i\}_{h_i\in H}.$$
We say that two elements $x,\tilde{x}\in \N G$ are right $H$-equivalent, and denote
$(x,\tilde{x})\in R_H$ if they
belong to the same right $H$-coset of $\N G$. We denote the equivalence class of $x\in\N G$ by $\bar{x}\in\N G/R_H$.
\end{definition}
We will also need the following immediate extension of group homomorphisms to morphisms of the corresponding semirings.
\begin{lemma}
Let $\phi:G_1\to G_2$ be a group homomorphism, then its extension
\begin{equation}\label{extended}
\begin{array}{rcl}
\bar{\phi}:\N G_1& \to & \N G_2\\
\sum _{g\in G_1} n_gg  & \mapsto & \sum _{g\in G_1} n_g\phi(g)
\end{array}
\end{equation}
is a morphism of semirings.
\end{lemma}

\subsection{Graded algebras}
A {\it graded-morphism} between two graded algebras
\begin{equation}\label{eq:equivgr}
\mathcal{G}_A:A=\bigoplus _{g\in G} A_g,\quad \mathcal{G}_B:B=\bigoplus _{h\in H} B_h,
\end{equation}
is a pair $(\psi,\phi)$, where $\psi:A\rightarrow B$ is an algebra homomorphism and $\phi:G\rightarrow H$
is a group homomorphism such that $\psi(A_g)\subset B_{\phi(g)}$ for any $g\in G$.
If both $\psi$ and $\phi$ are injective, we can regard $A$ as a {\it graded subalgebra} of $B$.

The notions of isomorphism and equivalence of graded vector spaces extend to graded algebras.
\begin{definition}\label{greqdef}
\cite[P. 685, Remark]{BSZ}
A graded-morphism $(\psi,\phi)$ between two graded algebras \eqref{eq:equivgr} is a {\it graded-equivalence} if
$\psi$ and $\phi$ are isomorphisms of algebras and groups respectively.
It is a {\it graded-isomorphism} if it is a graded-equivalence and,
additionally, $H=G$ and $\phi$ is the identity map.
\end{definition}
We usually adopt the first equivalence relation, which fits our context well (see \S \ref{fundint}).
The difference between the above notions is manifested by twisted group algebras, and can easily
be verified by a direct calculation.
\begin{proposition}\label{prop:equivtw}(see \cite[Corollary 2.22]{EK13})
Let $\C^{f_1}G$ and $\C^{f_2}G$ be two twisted group algebras with the natural grading. Then
\begin{enumerate}
\item $\C^{f_1}G$ and $\C^{f_2}G$ are graded-isomorphic if and only if $[f_1]=[f_2]\in H^2(G,\C^*)$.
\item $\C^{f_1}G$ and $\C^{f_2}G$ are graded-equivalent if and only if $[f_1]$ and $[f_2]$ lie in the same
Aut$(G)$-orbit in $H^2(G,\C^*)$ (see~\eqref{eq:action}).
\end{enumerate}
\end{proposition}
Example \ref{example} is a good way to understand Proposition \ref{prop:equivtw}. Note that any nonzero
$u_{\sigma}\in B_{\sigma}\backslash \{0\}$ and $u_{\tau}\in B_{\tau}\backslash \{0\}$ are invertible and satisfy
\begin{equation}\label{eq:unity}
u_{\tau}u_{\sigma}u_{\tau}^{-1}u_{\sigma}^{-1}= \xi _dI_d.
\end{equation}
Since graded-isomorphism preserve~\eqref{eq:unity}, then replacing the $d$-\textit{th} primitive root of unity
$\xi _d$ by a different primitive root of unity yields a non-isomorphic grading.
However, for another primitive root of unity $\xi _d^r$ (where $r$ and $d$ are coprime),
the (non-identity) group automorphism
\begin{equation}\label{eq:cdcd}
\begin{array}{rcl}
\phi:C_d\times C_d& \to & C_d\times C_d\\
\sigma   & \mapsto & \sigma \\
\tau   & \mapsto & \tau ^r
\end{array}
\end{equation}
determines a grading-equivalence between the corresponding twisted group algebras.\qed

\begin{definition}\label{morphismclass}
Let \eqref{eq:equivgr} be two graded algebras. Two graded-morphisms
$$(\psi_1,\phi_1):\mathcal{G}_A\to \mathcal{G}_B,\ \ (\psi_2,\phi_2):\mathcal{G}_A\to \mathcal{G}_B$$
are {\it equivalent} if there exist two graded-equivalences
$$(\psi_3,\phi_3):\mathcal{G}_A\to \mathcal{G}_A,\ \ (\psi_4,\phi_4):\mathcal{G}_B\to \mathcal{G}_B,$$
such that $$(\psi_2\circ\psi_3,\phi_2\circ\phi_3)=(\psi_4\circ\psi_1,\phi_4\circ\phi_1).$$
The equivalence class of the graded-morphism $(\psi_1,\phi_1)$ can be written as
$$[(\psi_1,\phi_1)]:[\mathcal{G}_A]\to [\mathcal{G}_B].$$
\end{definition}

The following structure, namely the {\it free product grading}, will by useful later on.
For every $j=1,\cdots,r$, let $\mathcal{G}_j$ be a grading of an algebra $A_j$ by a group $\Gamma_j$ (with a trivial element $e_j$).
Then the direct sum $A=\bigoplus_{j=1}^rA_j$ is endowed with a natural grading $\coprod_{j=1}^r \mathcal{G}_j$ by the free product $\Gamma=\coprod_{j=1}^r \Gamma_j$
determined by the homogeneous components as follows.
\begin{equation}\label{coprodgrade}
\coprod_{j=1}^r \mathcal{G}_j:
\begin{array}{ll}
A_e=\bigoplus_{j=1}^r(A_j)_{e_j},&\\
A_{g_j}=(A_j)_{g_j},&  e_j\neq g_j\in \Gamma_j,\\
A_g=0,& g\notin \Gamma_j, j=1,\cdots,r.\end{array}
\end{equation}
The following claim is straightforward.
\begin{lemma}\label{directsum}(see \cite[Lemma 6.4]{cibils2010})
If the $\Gamma_j$-gradings $\mathcal{G}_j$ of $A_j$ are connected for all $j=1,\cdots,r$, then so is the $\coprod_{j=1}^r \Gamma_j$-grading $\coprod_{j=1}^r \mathcal{G}_j$ of $A$.
\end{lemma}
We denote the grading class of $\coprod_{j=1}^r \mathcal{G}_j$ by $\coprod_{j=1}^r [\mathcal{G}_j]$, where $[\mathcal{G}_j]$ is the grading class of $\mathcal{G}_j$ for $j=1,\cdots,r$.

\subsection{Quotient gradings}\label{qg}
\begin{definition}\label{def:quoalg}
A graded-morphism $(\psi,\phi)$ between two graded algebras \eqref{eq:equivgr} is a {\it quotient morphism} if
$\psi$ is an isomorphisms of algebras, and the group homomorphism $\phi$ is onto.
In this case we say that $\mathcal{G}_B:B=\bigoplus _{h\in H} B_h$ is a {\it quotient-grading} of
$\mathcal{G}_A:A=\bigoplus _{g\in G} A_g$.
\end{definition}
Quotients of graded algebras are often given just by a quotient of the grading group. Let \eqref{eq:algebragrading} be a $G$-graded algebra and let $N\lhd G$ be a normal subgroup of $G$.
Then by a $G/N$-grading of $A$ we mean the quotient-grading of the algebra $A$ given by
$$A=\bigoplus _{\bar{g}\in G/N} A_{\bar{g}},$$
where $$ A_{\bar{g}}:=\bigoplus_{g\in\bar{g}}A_g.$$
Such quotient morphisms can be denoted by pairs $(\psi,\phi)=($Id$_A,$ mod $N).$

A special quotient morphism is singled out for later use. It is the graded-morphism which ``forgets" the graded data of \eqref{eq:algebragrading} and leaves only its algebra structure
(i.e., the quotient grading is by the trivial group).
This quotient morphism, noted as above for the subgroup $G\lhd G$ by $($Id$_A,$ mod $G)$, is called the {\it forgetful} morphism.
Using definition \ref{morphismclass}, we may also deal with the equivalence class of this forgetful morphism.

Clearly, any graded-isomorphism and graded-equivalence are quotient morphisms.
Moreover, quotient morphisms determine a well-defined relation on the set of graded-isomorphism classes and on the set of
graded-equivalence classes of a given algebra (that is, if a member of one class is a quotient grading of a member of the second class). We can therefore refer to a grading class
as a quotient of another grading class.
When the algebra is finite-dimensional, this relation, restricted to classes of connected gradings, is a partial order.
\begin{proposition}\label{poset}
Under the quotient relation, the connected graded-isomorphism classes and the connected
graded-equivalence classes of a finite-dimensional algebra are partially ordered sets.
\end{proposition}
\begin{proof}
We prove the claim for connected graded-equivalence classes. The proof for connected graded-isomorphism classes is similar.
Reflexivity and transitivity are clear (even when the algebras are infinite-dimensional, or when the grading classes are not connected). We check antisymmetry. Let
\begin{equation}\label{inversequotients}
\begin{array}{cc}
\psi_1:A\to B, & \phi_1:G\to H\\
\psi_2: B\to A, & \phi_2:H\to G
\end{array}
\end{equation}
be algebra isomorphisms and surjective group homomorphisms (respectively) such that $(\psi_1,\phi_1)$ and $(\psi_2,\phi_2)$
are quotient morphisms between the graded algebras \eqref{eq:equivgr}. We show that both are graded-equivalences.
The group morphism $$\phi_2\circ \phi_1: G\to G$$ is surjective, hence the finite dimension condition yields that
$\phi_2\circ \phi_1$ is bijective on the finite set supp$_G(A)$.
Denote the permutation induced by $\phi_2\circ \phi_1$ on the support supp$_G(A)$ by $\sigma$, and let $m$ be its order.
Now, let $$g\in \text{ker}(\phi_2\circ \phi_1).$$
By the connectivity condition, we can write
$$g=\prod _{i=1}^rg_i^{t_i},\ \ g_i\in\text{supp}_G(A),\ \ t_i\in\mathbb{Z}.$$ Then
\begin{equation}\label{eq:ginkernal}
e=\phi_2\circ \phi_1(g)=\prod _{i=1}^r\phi_2\circ \phi_1(g_i)^{t_i}=\prod _{i=1}^r\sigma(g_i)^{t_i}.
\end{equation}
Applying $\phi_2\circ \phi_1$ iteratively, ~\eqref{eq:ginkernal} yields
$$e=(\phi_2\circ \phi_1)^m(g)=\prod _{i=1}^r\sigma^m(g_i)^{t_i}=\prod _{i=1}^rg_i^{t_i}=g.$$
We conclude that $\phi_2\circ \phi_1:G\to G$ is injective, and hence so is $\phi _1:G\to H$.
Consequently, $\phi _1$ is a isomorphism and therefore $(\psi_1,\phi_1)$
(as well as $(\psi_2,\phi_2)$) is a graded equivalence.
\end{proof}
\begin{definition}\label{scdiagdef}
Let $A$ be a finite-dimensional algebra.
\begin{enumerate}
\item The diagram of isomorphism types of groups that are associated to the connected grading classes of $A$
with the corresponding quotient morphisms is denoted by $\Delta(A)$.
\item When we say that a connected grading class of $A$ is {\it maximal connected}, or just {\it maximal}, it is under the quotient partial order.
\end{enumerate}
\end{definition}

\subsection{The intrinsic fundamental group}\label{fundint}
The intrinsic fundamental group, introduced in \cite{cibilsintrinsic} for any linear category,
can be defined for finite-dimensional algebras $A$, regarded as 1-object categories, using the diagram $\Delta(A)$ (Definition \ref{scdiagdef}(1)).
Here is a brief summary.

Any covering of linear categories
\begin{equation}\label{concov}
F:\mathcal{C}\rightarrow \mathcal{B}
\end{equation}
gives rise to a group Aut$_1(F)$ \cite[Corollary 2.10]{cibilsintrinsic}
which acts on the objects of the category $\mathcal{C}$.
The covering \eqref{concov} is called {\it Galois} if the category $\mathcal{C}$ (and hence also $\mathcal{B}$) is connected,
and the action of Aut$_1(F)$ on every fiber is transitive.

Any connected $G$-grading of $\mathcal{B}$ determines a connected {\it smash-product} category $\mathcal{B}\#G$
\cite[Definition 2.6]{cibils2010} as well as a natural Galois covering $F:\mathcal{B}\# G\rightarrow \mathcal{B}$ such that Aut$_1(F)\cong G$.
Conversely, let \eqref{concov} be a Galois covering. Then there is a natural connected Aut$_1(F)$-grading on $\mathcal{B}$
such that the natural Galois covering $$\mathcal{B}\#\text{Aut}_1(F)\rightarrow \mathcal{B}$$ is equivalent to $F$ (see \cite[Theorem 5.3, Remark 5.4]{cibilsintrinsic}).
Let
$$F_i:\mathcal{C}_i\rightarrow \mathcal{B}, \ \ i=1,2$$
be two Galois coverings of $\mathcal{B}$. Then there exists a morphism $F_1\rightarrow F_2$ of Galois coverings
if and only if the corresponding Aut$_1(F_2)$-grading of $\mathcal{B}$ is a quotient grading of the Aut$_1(F_1)$-grading of $\mathcal{B}$.
The above yield an equivalence between the category Gal$(\mathcal{B})$ of Galois coverings of $\mathcal{B}$ up to
equivalence and the category of
connected gradings of $\mathcal{B}$ (up to graded-equivalence, see Definition \ref{greqdef}) with the quotient morphisms.
By \cite[Definition 8 and \S 4 ]{cibils2012universal} or \cite[Proposition 2.10]{cibils2011}
we can adopt the following as a definition.
\begin{definition}\label{prop:grint}
The intrinsic fundamental group $\pi_1(A)$ of a finite-dimensional $\mathbb{C}$-algebra $A$
is the inverse limit of the diagram $\Delta(A)$.
\end{definition}

Maximal connected gradings as well as the intrinsic fundamental group may sometimes be established by free products \eqref{coprodgrade} as suggested by the following lemma.
\begin{lemma}\label{+free}
Suppose that any connected grading of a finite-dimensional algebra $A=A_1\oplus A_2$ admits both $A_1$ and $A_2$ as graded ideals.
Then the maximal connected grading classes of $A$ are exactly the free product
grading classes $[\mathcal{G}_1]*[\mathcal{G}_2]$, where $[\mathcal{G}_1]$ and $[\mathcal{G}_2]$ run over the maximal connected grading classes of $A_1$ and $A_2$ respectively.
Consequently, in this case $\pi_1(A)=\pi_1(A_1)*\pi_1(A_2).$
\end{lemma}
\begin{proof}
The fact that $A_1$ and $A_2$ are graded ideals of any connected grading of $A$ implies at once that any connected grading class of $A$ is a quotient of a free product of maximal connected classes
of $A_1$ and $A_2$. It also implies that if $[\mathcal{G}_1]$ and $[\mathcal{G}_2]$ are maximal connected grading classes of $A_1$ and $A_2$ respectively, then
a free product $\mathcal{G}_1*\mathcal{G}_2$ is a quotient of a connected grading $\mathcal{G}$ of $A$ only under an equivalence grading morphism.
Hence, to establish the fundamental group of $A$, the diagram $\Delta(A)$ may be restricted to free products between groups that grade $A_1$ and $A_2$.
The inverse limit of this diagram is therefore the free product of the corresponding inverse limits.
\end{proof}

\subsection{Induction}\label{inductionsec}
Let $V$ be a complex vector space and let $V^*=$Hom$(V,\C)$ be its dual.
As customary, we identify the tensor $v\otimes \phi \in V\otimes V^*$ with an endomorphism of $V$ by
$$ (v\otimes \phi )(w)=\phi(w)\cdot v, \quad\forall w\in V.$$
Composition $(v_1\otimes \phi _1) \circ (v_2\otimes \phi _2)$ of such endomorphisms is given by
$$[(v_1\otimes \phi _1) \circ (v_2\otimes \phi _2)](w)=(v_1\otimes \phi _1)(\phi _2(w)v_2)=\phi _2(w) \phi _1(v_2)v_1.$$
Then $$(v_1\otimes \phi _1) \circ (v_2\otimes \phi _2)=\phi _1(v_2)(v_1 \otimes \phi _2).$$
Let $V$ be a finite-dimensional vector space.
In this case, the above identification yields an algebra isomorphism between $V\otimes V^*$ and End$_{\C}(V)$.
If $V$ is also equipped with a $G$-grading \eqref{gradvecsp},
then the tensor product grading \eqref{eq:tensvec} yields a natural $G$-grading on End$_{\C}(V)$ as an algebra.
The following notion was suggested in \cite[p.3]{das1999}, see also \cite{CDN02}.
\begin{definition}\label{goodef}
Let \eqref{gradvecsp} be a $G$-grading of an $n$-dimensional vector space.
Then the tensor product $G$-grading \eqref{eq:tensvec} of $V\otimes V^*\cong$End$_{\C}(V)(\cong M_n(\C))$ is called {\it elementary}.
\end{definition}
This way of grading can be pushed further as follows.
Let \eqref{eq:algebragrading} be a $G$-graded algebra
and let \eqref{gradvecsp} be a finite-dimensional $G$-graded vector space. The
algebra structure of
\begin{equation}\label{eq:vecinduce}
V\otimes A \otimes V^*\cong A\otimes\text{End}_{\C}(V)(\cong M_{\text{ dim}_{\C}(V)}(A))
\end{equation}
is determined by the multiplication
\begin{equation}\label{defmultens}
(v_{g_1}\otimes a_{h_1} \otimes \phi _{g_2})\circ (v_{g_3}\otimes a_{h_2} \otimes \phi _{g_4})=
\phi_{g_2}(v_{g_3})(v_{g_1}\otimes a_{h_1h_2}\otimes \phi_{g_4}).
\end{equation}
Then, the induced vector space $G$-grading \eqref{indef} on $V\otimes A \otimes V^*$ respects the
algebra structure \eqref{defmultens}.
We denote this induced $G$-grading on the algebra $V\otimes A \otimes V^*$ by $x(A),$ where $x=\chi(\mathcal{G}) \in \mathbb{N}G.$
We think of $A$ as a graded subalgebra of $x(A)$ by identifying it with $A\otimes $ Id$(V)$, where Id$(V)$ denotes the identity endomorphism of $V$.
Our focus will mainly be on the algebras $M_n(A)$ without caring too much about the algebra isomorphism \eqref{eq:vecinduce}, hence by abuse of notation,
$x(A)$ will also denote the corresponding grading class.
\begin{definition}\label{algindef}
A $G$-grading $Y$ of an algebra $B$ is {\it induced}
from a $G$-grading~\eqref{eq:algebragrading} of an algebra $A$ if there exists $x \in \mathbb{N}G$
such that $Y$ is graded isomorphic to $x(A)$.
\end{definition}
Let~\eqref{eq:algebragrading} be a $G$-grading of an algebra $A$ and let $\sum _{g\in G} n_gg\in \mathbb{N}G$.
By \eqref{eq:tensvec},~\eqref{dual} and~\eqref{indef} the support of the induced $G$-grading $\sum _{g\in G} n_gg(A)$ is given by
\begin{equation}\label{eq:suppinduce}
\text{supp}_G(\sum _{g\in G}n_gg(A))=\{g_1g_2g_3^{-1}|g_2\in \text{supp}_G(A),\ \ n_{g_1},n_{g_3}>0\}.
\end{equation}

Graded morphisms respect induction. Let $(\psi,\phi)$ be a graded morphism from the $G$-graded algebra $x(A)$
to an $H$-graded algebra $B$.
In particular the graded subalgebra $A$ maps to a graded subalgebra $\psi(A)\subset B$.
Then the image $\psi(x(A))$ is a graded subalgebra of $B$, which is graded equivalent to the induced algebra
$\bar{\phi}(x)(\psi(A))$, where $\bar{\phi}(x)$ is the natural extension of $\phi$ to $\mathbb{N}G$ (see \eqref{extended}).
We record this observation for the case of quotient morphisms using Definition \ref{cosetequiv}.

\begin{lemma}\label{lemma:quoinduce}
Let $A$ be a $G$-graded algebra, $x\in \N G$, and $N\lhd G$ a normal subgroup.
The $G/N$-quotient grading of the induced $G$-graded algebra $x(A)$ is graded isomorphic to
$\bar{x}(B)$ ,where $\bar{x}\in \N (G/N)$ is the equivalence class of $x$
in $\N G/R_N$, and $B$ is the algebra $A$ with the corresponding $G/N$-grading.
In particular, a quotient grading of an elementary grading is elementary.
\end{lemma}

Recall the augmentation map
\begin{equation}\label{eq:augm}
\begin{array}{rcl}
\epsilon:\mathbb{N}G& \to & \mathbb{N}\\
\sum _{g\in G} n_gg   & \mapsto & \sum _{g\in G} n_g,
\end{array}
\end{equation}
which is just the extension of the trivial group morphism $G\to \{e\}$ to a morphism of semirings (see \eqref{extended}).

Definitions \ref{goodef} and \ref{algindef} yield
\begin{corollary}\label{elem}
A $G$-grading of the algebra $M_n(\C)$ is elementary
if and only if it is induced from the trivial grading on $\C$ by an element
$x\in \mathbb{N}G$ with $\epsilon(x)=n$.
\end{corollary}
By Corollary \ref{elem}, any element $x\in \mathbb{N}G$ gives rise to a unique (up to $G$-graded isomorphism) elementary grading on $M_{\epsilon(x)}(\C)$
, namely the endomorphism algebra of any $G$-graded space affording $x$ as a character.
In general, this correspondence between $\N G$ and the $G$-graded isomorphism classes of elementary gradings is not
injective, see Theorem \ref{th:AHequi}.

A special elementary $G$-grading is obtained when $G$ is finite. Recall the distinguished $G$-graded isomorphism class
$[\mathcal{G}_0]\in$$\mathcal{S}_G$, whose character is the sum of all elements of $G$ (see \S\ref{cps}).
The graded endomorphism algebra of a graded vector space in this class has a special notation itself:

\begin{definition}\label{def:cpg} (see \cite[Definition 3]{ak13})
An elementary $G$-grading of $M_n(\C)$ is called a {\it crossed product $G$-grading}
if it is induced from the trivial grading on $\C$ by $\sum _{g\in G} g\in \mathbb{N}G$.
\end{definition}
As in the graded vector spaces situation, the set of graded isomorphism classes and the set
of graded equivalence classes of finite dimensional $G$-graded algebras have a structure of an $\mathbb{N}G$-semimodules
with the induction action.
Clearly, the family of finite-dimensional semisimple algebras is closed under direct sums.
Moreover, simplicity and semisimplicity are preserved under induction, and therefore
the set of graded isomorphism classes and the set of graded equivalence classes
of finite-dimensional $G$-graded semisimple algebras are $\mathbb{N}G$-sub-semimodules of the above semimodules respectively.

\subsection{Graded-simple algebras}\label{gsa}

\begin{definition}
A {\it graded ideal} of a group-graded algebra is a (two-sided) ideal which is a graded-subspace.
A graded algebra is {\it graded-simple} if it admits no non-trivial graded ideals.
\end{definition}

In order to understand gradings of semisimple algebras, it is enough to restrict to simple gradings.
This is a consequence of the following graded-decomposition theorem.
\begin{theorem}\label{prop:decsemsi}(see e.g. \cite[Theorem 2.3']{CM}, \cite[\S A.$1$]{NVO82})
Any group-grading of a semisimple finite-dimensional algebra admits a decomposition as a direct sum of graded-simple algebras.
\end{theorem}

Twisted group algebras $\C ^fG$ are natural examples of $G$-simple gradings.
Indeed, any nonzero element in $(\C ^fG)_g=\text{span}_{\C}\{u_g\}$ is invertible.
Therefore, a nonzero graded subspace of $\C ^fG$ always contains an invertible element, and hence the twisted group algebra
does not admit nontrivial graded ideals (even one-sided ones).

It is easy to show that a group-graded algebra which is induced from a graded-simple algebra
is graded-simple as well. In particular, any grading which is induced from a twisted group algebra $\C ^fG$
is graded-simple. Bahturin, Sehgal and Zaicev prove that the converse is also true.
\begin{theorem}(see \cite[Theorem 5.1]{MR1941224},\cite[Theorem 3]{MR2488221})\label{th:BSZ}
Let $$\mathcal{G}:A=\bigoplus _{g\in \Gamma} A_g$$
be a $\Gamma$-simple grading. Then there
exists a subgroup $G<\Gamma$, a 2-cocycle $f\in Z^2(G,\C^*)$ and $x\in
\N \Gamma$ with $$|G|\cdot\epsilon(x)^2=\text{dim}_{\C}(A),$$ such that
$\mathcal{G}$ is graded-isomorphic to $x\left(\C^fG\right).$ In particular, $A$ is simple if and only if $f$ is
non-degenerate.
\end{theorem}
Theorem \ref{th:BSZ} (together with Theorem \ref{prop:decsemsi})
says that the $\mathbb{N}\Gamma$-semimodule of graded-isomorphism classes of finite-dimensional
semisimple $\Gamma$-graded algebras is generated by the graded isomorphism classes of all the twisted group algebras $\C^fG$,
where $G<\Gamma$ and $f\in Z^2(G,\C^*)$. The same holds for the corresponding semimodule of graded equivalence classes.

Using Definition \ref{cosetequiv} we can formulate the conditions for graded simple algebras
to be graded-equivalent and graded-isomorphic.
This is a generalization of Proposition \ref{prop:equivtw}.
\begin{theorem}\label{th:AHequi}(see \cite[Lemma 1.3, Proposition 3.1]{aljadeff2011simple},\cite[Corollary 2.22]{EK13})
Let $x\left(\C ^{f_1}H_1\right)$ and $\tilde{x}\left(\C ^{f_2}H_2\right)$ be a $G_1$-grading and a $G_2$-grading of two algebras $A_1$ and $A_2$ respectively.
Then $A_1$ and $A_2$ are graded-equivalent if and only if there exist
\begin{enumerate}
\item group isomorphism $\phi :G_1\rightarrow G_2$,
\item an element $g\in G_1$
\end{enumerate}
 such that
$$(\phi^{-1}(\tilde{x})g^{-1},x)\in R_{H_1},\quad  H_2=\phi(gH_1g^{-1}),\quad[f_2]=[\phi(gf_1g^{-1})].$$
If $G_1=G_2$, then $A_1$ and $A_2$ are graded-isomorphic if and only if with the same notation $\phi$ is the identity on $G$.
\end{theorem}
\begin{remark}\label{1inx}
Taking an appropriate $g\in G_1$ in Theorem \ref{th:AHequi}(2), one can ``normalize" the semiring element $x$ to be of the form $1+y$ for some $y\in\N G_1$ with $\epsilon(y)=\epsilon(x)-1$.
\end{remark}
\begin{remark}
Note that by Theorem \ref{th:AHequi}, an elementary grading can be graded-equivalent only to an elementary grading.
Furthermore, $x,\tilde{x}\in \N G$ give rise to the same elementary $G$-graded isomorphism class if and only if $\tilde{x}=xg$
for some $g\in G$.
In particular, $\sum _{g\in G}g$ is the only element in $\N G$ which gives rise to the crossed product $G$-grading (see Definition \eqref{def:cpg}).
\end{remark}

\subsection{Terminology}\label{terminology}
This work maneuvers between the different terminologies used in the literature.
Here is a comparison between our main sources.

Firstly, a finite group $G$ is classically \cite{isaacs}
said to be of central type if it has a faithful irreducible representation
of dimension $\sqrt{|G:Z|}$, where $Z$ is the center of $G$.
In our paper as well as in other references, e.g. \cite{aljadeff}, \cite[page 35]{EK13},
the term ``groups of central type" is used for the quotient of a (classical) group of central type by its center.
Such groups admit irreducible {\it projective} representations of dimension that is the square root of their order.
In \cite{Gagola} and in \cite{ShSh} such groups are termed {\it fully ramified} and {\it central-type factor groups} respectively.

A class of $f$-regular conjugated elements (see Definition \ref{regularitydef}) is termed an {\it $f$-ray class} in \cite{Oystaeyen}.

Our notions of {\it graded-equivalence} and {\it graded-isomorphism} follow the footsteps of \cite{BSZ}.
While graded-isomorphism has the same
meaning in \cite[Definition 1.15]{EK13}, graded-equivalence is termed {\it weak isomorphism} of gradings \cite[page 16]{EK13}.
The term equivalence of gradings is reserved in \cite[Definition 1.2]{EK12} and in \cite[page 14]{EK13} for a relation which is weaker than both of these two.

A grading such that the dimension of all the homogeneous components is at most 1 is called {\it fine} in
\cite[page 1]{MR2226177} and \cite[Definition 2.1]{MR1941224},
whereas in \cite[page 18]{EK13} fine gradings are justifiably defined more generally to admit no proper refinement
(see the remark following Definition 2.2.1 in \cite{EK13}).

In \cite{das2008}, connected gradings are called {\it faithful}. The term {\it simply-connected gradings} is used in \cite[Definition 4.3]{cibils2010} for
connected gradings whose equivalence grading class is maximal with respect to the quotient
partial ordering (or just maximal-connected, see Definition \ref{scdiagdef}(2)).
Such $G$-gradings, in the terminology of \cite{EK13}, are the fine $G$-gradings whose universal grading
group (see \cite[Definition 1.4]{EK12}) is $G$ itself.

Finally, elementary gradings are called {\it good} in \cite{CDN02, cibils2010, das2008}.

\section{Families of groups of central type}
In this section we deal with certain groups of central type.
\S\ref{perct} gives some background about the alternating form that is associated to every cohomology class, and \S\ref{Hall} addresses preliminary results about Hall subgroups.
\S\ref{autact} discusses a generalization of the concept of symplectic actions and symplectic groups.
Theorem \ref{prop:newctsdpr}, presenting a necessary and sufficient condition for the closure of the family of groups of central type to certain extensions of Hall subgroups,
and Theorem \ref{thatype}, describing groups of central type whose Sylow subgroups are abelian, are given in this subsection.
Theorem A and Corollary B, describing all groups of central type of cube-free order, are proven in \S \ref{Proof of Theorem C}.
\S \ref{non-transitive} contains examples of groups of central type, the automorphism groups of which do not act transitively on their non-degenerate cohomology classes.

\subsection{The associated alternating form}\label{perct}

To any 2-cocycle $f\in Z^2(G,\C^*)$ one associates a complex {\it alternating $G$-form} (see \cite[\S 1.2]{david2013isotropy})
\begin{equation}\label{altform}
\begin{array}{rcl}
\alpha_f:G\times G &\rightarrow &\C^*\\
(g,h)&\mapsto & f(h,g)\cdot f(hgh^{-1},h)^{-1}.
\end{array}
\end{equation}
This $G$-form satisfies the following properties.
\begin{enumerate}
\item $\alpha_f(g,h)=\alpha_f(h,g)^{-1}$ for every $(g,h)\in\mathcal{A}_G,$ where
$$\mathcal{A}_G:=\{(g,h)\in G\times G| gh=hg\}.$$
\item For every $g\in G$, the map
\begin{eqnarray}\label{linchar}
\begin{array}{rcl}
C_G(g)&\to& \C^*\\
h&\mapsto &\alpha_f(g,h)
\end{array}
\end{eqnarray}
is a linear character of the centralizer $C_G(g)$ of $g$.
\item If $[f']=[f]\in H^2(G,\C^*)$ then
$$\text{res}|_{\mathcal{A}_G}^{G\times G}\alpha_{f'}=\text{res}|_{\mathcal{A}_G}^{G\times G}\alpha_f.$$
That is, the $G$-form depends only on the cohomology class of $f\in Z^2(G,\C^*)$ when restricted to $\mathcal{A}_G$.
The restricted form is therefore associated to a cohomology class and can also be written as $\alpha_{[f]}$.
\end{enumerate}
By properties (1) and (2) above we get
\begin{lemma}\label{altformorder}
For every $(g,h)\in\mathcal{A}_G,$ $\alpha_f(g,h)$ is a root of 1 of order dividing both orders of $g$ and $h$.
\end{lemma}

When a group $A$ is abelian, $\mathcal{A}_A$ coincides with
$A\times A$ and we can regard $\alpha_f$ as a complex {\it alternating
bicharacter} (see \cite[\S 2.2]{EK13}). Suppose that a finite abelian group is given by
\begin{equation}\label{eq:decompofabeliangroupptok}
A=\langle x_1\rangle\times \langle x_2\rangle \times \ldots \times\langle x_k\rangle\cong
C_{n_1}\times C_{n_2}\times \ldots \times C_{n_k},
\end{equation}
such that $n_i$ divides $n_{i+1}$ for any $1\leq i\leq k-1$.
Then any $[f]\in H^2(A,\C ^*)$ is determined by the
$\frac{k(k-1)}{2}$ values
\begin{equation}\label{formab}
\alpha_f(x_i,x_j)\ \ {1\leq i<j\leq k},
\end{equation}
where $\alpha_f(x_i,x_j)$ is an $n_i$-th root of unity \cite{Schur}.
In this abelian case, the $A$-form
$\alpha_f$ uniquely determines the class $[f]\in H^2(A,\C^*)$.
This is not the case in general, that is, for a finite group $G$, the well-defined map
$[f]\mapsto \alpha_f$ is not necessarily injective (it admits the
Bogomolov multiplier $B_0(G)<H^2(G,\C^*)$ as a kernel \cite{Bogo}).
Nevertheless, we will often use the $G$-form notation rather than
deal with a cohomology class even if it corresponds to more than
one class.
\begin{definition}\label{regularitydef}
Let $f\in Z^2(G,\C^*)$. An element $g\in G$ is {\it $f$-regular} if the $C_G(g)$-linear character \eqref{linchar} is trivial.
\end{definition}
It can easily be verified that if an element is $f$-regular, then so are all its conjugated elements.
\begin{theorem}\label{rayclasses}(see \cite[Theorem 2.4]{Oystaeyen})
Let $f\in Z^2(G,\C^*)$. Then the dimension of the center of the twisted group algebra $\C^fG$ is equal to the number of $f$-regular conjugacy classes in $G$.
\end{theorem}
Theorem \ref{rayclasses} gives another way to say that $f$ is
non-degenerate.
\begin{corollary}\label{regularity}
A cocycle $f\in Z^2(G,\C^*)$ is non-degenerate if and only if there is no nontrivial $f$-regular conjugacy class, that is res$|^G_{C_G(g)}\alpha_f(g,-)=1$ implies $g=e$.
\end{corollary}
This non-degeneracy property of the $G$-form $\alpha_f$ on $\mathcal{A}_G$ deserves naming it a
{\it symplectic} group form.
Symplectic forms over finite abelian groups can be described as follows.
\begin{theorem}\label{abct}\cite[Definition 3, Theorem 5]{BSZ},\cite[Corollary 3.10]{ShSh}
Abelian groups of central type are exactly groups of the form $A\times A$,
that is the direct sum of two copies of the same abelian group $A$.
Furthermore, an $A\times A$-form $\alpha_f$ is symplectic if and only if
there is a decomposition \eqref{eq:decompofabeliangroupptok} of $A$ such that
\begin{align}\label{eq:ndcohonab}
\alpha_f((x_i,1),(x_j,1))=\alpha_f((1,x_i),(1,x_j))=1,\ \ \alpha_f((x_i,1),(1,x_j))=\delta_{i,j}\xi _i,
\end{align}
where $\xi_i$ is a root of unity of order $n_i$ and $\delta_{i,j}$ is Kronecker's delta.
\end{theorem}

\subsection{Hall subgroups}\label{Hall}
The following notation will be of use.
Let $G$ be a finite group and let $d$ be an integer (usually, but not necessarily, a square-free divisor of the order of $G$).
Let $\hat{d}$ be the minimal divisor of $|G|$, such that $d$ and $\frac{|G|}{\hat{d}}$ are coprime.
Then $\hat{d}$ and $\frac{|G|}{\hat{d}}$ are also coprime.
We denote by $G_d$ and $G_{d'}$ a $\hat{d}$-Hall and a $\frac{|G|}{\hat{d}}$-Hall subgroups of $G$ respectively, if exist.

The Schur-Zassenhaus Theorem will be useful in the sequel.
\begin{theorem}\label{sz} \cite[Theorem 25]{zassenhaus58}
Let $G_d$ be a normal Hall subgroup of $G$. Then there exists a $d^{\shortmid}$-Hall subgroup $G_{d^{\shortmid}}<G$ such that
$$G=G_d\rtimes G_{d^{\shortmid}}.$$
\end{theorem}

The following three lemmas provide sufficient conditions for the existence of normal Hall subgroups in groups that are of our concern.
\begin{lemma}\label{cor:semi}
Let $G_p$ be an abelian $p$-Sylow subgroup of a finite group $G$. Suppose that
\begin{equation}\label{autcoprimeindex}
gcd(\text{Aut}(G_p),[G:G_p])=1,
\end{equation}
that is the order of Aut$(G_p)$ is prime to the index of
$G_p$ in $G$. Then there exists a $p^{\shortmid}$-Hall subgroup $G_{p'}\lhd G$ such that
$$G=G_{p'}\rtimes G_p.$$

In particular, suppose that $p$ is the
smallest prime divisor of $|G|$, and either
\begin{enumerate}
\item $G_p$ is cyclic, or
\item $G_p$ is an abelian group of rank $2$ and $6$ is not a divisor of $|G|$.
\end{enumerate}
Then $$G=G_{p'}\rtimes G_p.$$
\end{lemma}
\begin{proof}
We show that $G_p$ is contained in the center of its normalizer, that is
\begin{equation}\label{BNC}
G_p\leq Z(\mathcal{N}_G(G_p)).
\end{equation}
By ``Burnside's Normal Complement Theorem" (see \cite[Theorem 7.50]{rotman}) this condition implies that
$G_p$ has a normal complement in $G$.
Since $G_p$ is abelian, the condition \eqref{BNC} holds if and only if the natural homomorphism
$$\begin{array}{rcl}
\phi:\mathcal{N}_G(G_p) &\rightarrow &\text{Aut}(G_p)\\
\phi(c)(g)&=&cgc^{-1},\ \ c\in\mathcal{N}_G(G_p), g\in G_p
\end{array}$$
is trivial. Indeed, let $c\in \mathcal{N}_G(G_p)$ be an element of the normalizer of order $p^j\cdot r$,
where $(p,r)=1$. Then there exist $\alpha,\beta\in \mathbb{Z}$ such that $\alpha \cdot r+\beta \cdot p^j=1$. Hence $c=ab$,
where $a:=c^{\alpha \cdot r}$ is of order $p^j$ and $b:=c^{\beta \cdot p^j}$ is of order prime to $p$.
Since $a$ and $b$ are powers of $c$, both lie in $\mathcal{N}_G(G_p)$. We show that they belong to the kernel of $\phi$, proving that so does $c$.
First, an element whose order is a power of $p$, and which normalizes a $p$-Sylow subgroup $G_p$ lies in $G_p$ (see \cite[Lemma 4.11]{rotman}). This proves that $a\in G_p<$ker$(\phi)$.
Next, the order of $\phi(b)$ divides the order of $b$ which is prime to $p$, and hence divides the index $[G:G_p]$. By the hypothesis \eqref{autcoprimeindex}, it is also prime to the order of Aut$(G_p)$
implying that $\phi(b)$ is trivial.

To prove the claim about the smallest prime divisor of $|G|$, we show that \eqref{autcoprimeindex} holds in both cases (1) and (2). Indeed,
\begin{enumerate}
\item If $G_p$ is cyclic of order $p^n$ then $|$Aut$(G_p)|=p^{n-1}(p-1)$.
Since $p$ is the smallest prime factor of $|G|$, no other factor divides $p-1$ yielding \eqref{autcoprimeindex}.
\item If $P$ is an abelian group of rank $2$, then (see, e.g. \cite{MR2411245}) there exist $i_1,i_2,i_3\in \mathbb{N}$ such that
\begin{equation}\label{eq:rank2}
|\text{Aut}(G_p)|=(p-1)^{i_1}p^{i_2}(p+1)^{i_3}.
\end{equation}
\end{enumerate}
Since $p$ is the smallest prime factor of $|G|$, the factors of $G$ do not divide $p-1$. Moreover, since $6$ is not a divisor of $|G|$, these factors also do not divide $p+1$ and we are done.
\end{proof}
\begin{remark}
If $G_p=C_{p^m}\times C_{p^l}$, where $m\neq l$ then $i_3= 0$ in~\eqref{eq:rank2} (see \cite{MR2411245}).
In these cases we can omit the condition that $6$ is not a divisor of $|G|$ in Lemma \ref{cor:semi}(2).
\end{remark}
\begin{lemma}\label{cor:sdpsdp}
Let $p<q$ be primes and let $G$ be a group of order dividing $p^2q^2$.

\begin{enumerate}
\item If $pq\neq 6$ then $G=G_q\rtimes G_p.$
\item If $|G|$ divides $36$ then either $G=G_2\rtimes G_3$ or $G=G_3\rtimes G_2.$
\end{enumerate}
\end{lemma}
\begin{proof}
The first part is a direct consequence of Lemma \ref{cor:semi}.
The case where $p=2$ and $q=3$ follows from Sylow's theorems together with Burnside's Normal Complement Theorem. It is left to the reader.
\end{proof}
Recall that by Hall's main theorem, a finite group is solvable if and only if it admits all possible Hall subgroups (see \cite[Theorem 5.28]{rotman}).
\begin{lemma}\label{prop:6primehall}
Let $G$ be a solvable group whose 2-Sylow and 3-Sylow subgroups are abelian of rank $\leq 2$.
Then $G$ admits a normal $6^{\shortmid}$-Hall subgroup. If, additionally, the order of $G_6$ divides 36,
then $G$ admits either a normal $2^{\shortmid}$-Hall subgroup or a normal $3^{\shortmid}$-Hall subgroup.
\end{lemma}
\begin{proof}
Let $G_{2^{\shortmid}}$ be a $2^{\shortmid}$-Hall subgroup of $G$.
By Lemma \ref{cor:semi}, a $3$-Sylow subgroup $G_3\leq G_{2^{\shortmid}}$ admits a normal complement.
Clearly, this normal complement is a $6^{\shortmid}$-Hall subgroup $G_{6^{\shortmid}}$ of $G$.
Similarly, a $2$-Sylow subgroup $G_2\leq G_{3^{\shortmid}}$ also admits a normal complement.
Again, this normal complement is a $6^{\shortmid}$-Hall subgroup $\tilde{G_{6^{\shortmid}}}$ of $G$.
By Hall's Theorem (see \cite[Theorem 5.28]{rotman}) $G_{6^{\shortmid}}$ and $\tilde{G_{6^{\shortmid}}}$ are conjugated in $G$.
Then, there exists an element $g\in G$ such that $G_{6^{\shortmid}}=g^{-1}\tilde{G_{6^{\shortmid}}}g$ and therefore
$g^{-1}G_{3^{\shortmid}}g$
is a $3^{\shortmid}$-Hall subgroup with a normal subgroup $G_{6^{\shortmid}}$.
Consequently, the normalizer of $G_{6^{\shortmid}}$ in $G$ contains both $g^{-1}G_{3^{\shortmid}}g$ and $G_{2^{\shortmid}}$.
Hence, the index $[G:\mathcal{N}_G(G_{6^{\shortmid}})]$ divides a power of $2$ and a power of $3$.
We conclude that $\mathcal{N}_G(G_{6^{\shortmid}})=G$ and hence $G_{6^{\shortmid}}$ is normal in $G$ (proving the first part of the claim)
yielding an onto morphism $\mu:G\to G_6$.

Suppose now that the order of $G_6$ divides 36. Then by Lemma \ref{cor:sdpsdp}(2), there is either an onto morphism $\mu_3:G_6\to G_3$
such that $$G_{3'}=\ker(\mu_3\circ\mu),$$ or an onto morphism $\mu_2:G_6\to G_2$ with $$G_{2'}=\ker(\mu_2\circ\mu).$$
This proves that at least one of these Hall subgroups is normal in $G$.
\end{proof}

Groups of central type are solvable \cite{isaacs,LY}
and as such admit all the possible Hall subgroups.
The following theorem will be useful in the sequel.
\begin{theorem}\label{th:hall}(see \cite[Corollary 4]{DeMeyer})
Let $G$ be a finite group and let $[f]\in H^2(G,\mathbb{C}^*)$. Then the following are equivalent.
\begin{enumerate}
\item The cohomology class $[f]$ is non-degenerate.
\item The restriction of $[f]$ to any Hall subgroup of
$G$ is non-degenerate.
\item The restriction of $[f]$ to any Sylow subgroup of
$G$ is non-degenerate.
\end{enumerate}
In particular, any Hall subgroup of a group of central type is itself of central type.
\end{theorem}

\subsection{The automorphism action on non-degenerate cohomology classes}\label{autact}
\begin{definition}
Let $H$ be a group and let $B$ be a subset of Aut$(H)$.
Denote the subgroup of $H^2(H,\mathbb{C}^*)$ consisting of
the $B$-invariant elements under the $B$-action~\eqref{eq:action}
by \textbf{$H^2(H,\mathbb{C}^*)^B$}. If
$\eta:G\rightarrow$Aut$(H)$ is an action of a group $G$ on
$H$, then we write
$$H^2(H,\mathbb{C}^*)^G:=H^2(H,\mathbb{C}^*)^{\eta(G)}.$$
\end{definition}
It is well known that if $B$ is a subset of the inner automorphisms of $H$ then
$$H^2(H,\mathbb{C}^*)^B = H^2(H,\mathbb{C}^*).$$
In other words, the action of Aut$(H)$ on $H^2(H,\mathbb{C}^*)$ is via the outer automorphisms of $H$.

Let $H$ be a subgroup of a group $G$. Then the normalizer
$\mathcal{N}_G(H)$ acts by conjugation on $H$.
By the above, the action of $\mathcal{N}_G(H)$ on
$H^2(H,\mathbb{C}^*)$ is via
\begin{equation}\label{weyl}
W(H):=\mathcal{N}_G(H)/H.
\end{equation}
The image of the restriction
$$\text{res}^G_H:H^2(G,\C^*)\to H^2(H,\C^*)$$
is contained in the subgroup $H^2(H,\mathbb{C}^*)^{W(H)}<H^2(H,\mathbb{C}^*)$ of elements which are fixed under the action
of this {\it Weyl group} $W(H)$ \cite[\S II.3]{AM04}.
When $H=G_p$ is a $p$-Sylow subgroup of $G$, the image of res$^G_{G_p}$ may coincide with these invariants,
as suggested by the following result, which is attributed to R.G. Swan.
\begin{theorem}\label{thmswan}\cite[Lemma 1]{swan60}
Let $G_p$ be an abelian $p$-Sylow subgroup of $G$. Then
$$Im(res_{G_p}^G)=H^2(G_p,\mathbb{C}^*)^{W(G_p)}.$$
\end{theorem}
When $N\lhd G$ is normal, then evidently  ${W(N)=G/N}$.
If, in addition, $G=N\rtimes K$ is finite and $\left(|N|,|K|\right)=1$, then
(see  \cite[Corollary 2.2.6]{karpilovsky1})
\begin{equation}\label{eq:semdircoh}
H^2(G,\mathbb{C}^*)=H^2(N,\mathbb{C}^*)^K\times H^2(K,\mathbb{C}^*).
\end{equation}

We are particularly interested in the restriction of the action \eqref{eq:action} to the subset of non-degenerate cohomology classes of $G$.
This subset is indeed Aut$(G)$-stable.
\begin{proposition}\label{stable}(see \cite[\S 4]{aljadeff2})
Let $[f]\in H^2(G, \C^*)$ and let $\phi\in$Aut$(G)$. Then $[f]$ is non-degenerate if and only if so is $\phi([f])$.
\end{proposition}
\begin{proof}
By Proposition \ref{prop:equivtw}(2), the twisted group algebras $\C^fG$ and $\C^{\phi(f)}G$ are graded-equivalent,
in particular, isomorphic as algebras. Consequently, one of them is simple if and only if so is the other.
\end{proof}

\begin{definition}\label{symplecticdef}
Let $[f]\in H^2(G,\C^*)$ be a non-degenerate cohomology class.
The {\it symplectic group} of $[f]$ is the stabilizer
$$\text{Sp}([f]):=\text{Stab}_{\text{Aut}(G)}([f])=\{\phi\in\text{Aut}(G)| \phi([f])=[f]\} $$
An action $\eta: K\to \text{Aut}(G)$ of a group $K$ on $G$ is symplectic if there exists a non-degenerate class $[f]\in H^2(G,\C^*)$ such that
$\eta(K)<\text{Sp}([f])$.
\end{definition}
By the above, it is clear that the symplectic groups
of all non-degenerate classes of $G$ contain the group Inn$(G)$ of inner $G$-automorphisms.
If the action of Aut$(G)$ on the non-degenerate cohomology classes in $H^2(G,\C^*)$ is transitive,
then the symplectic groups of all the non-degenerate classes are conjugate. This allows us to denote their isomorphism type as {\it the} symplectic group
of $G$. For example, the symplectic group of any abelian group is well defined. However, Examples \ref{example0}, \ref{example1} and
\ref{example2} demonsrtrate that this is not the case in general.

The symplectic group of an abelian group $G=C_m\times C_m$ of rank 2 is easily described.
We identify its automorphism group $\text{Aut}(C_m\times C_m)$ with GL$_2(m)$ by choosing two generators $C_m\times C_m=\langle\sigma,\tau\rangle$
and representing any automorphism
$$\begin{array}{rcl}
\phi &\in &\text{Aut}(C_m\times C_m)\\
\sigma &\mapsto &\sigma^{i_1}\tau^{j_1}\\
\tau &\mapsto &\sigma^{i_2}\tau^{j_2}
\end{array}$$
(where $i_1,j_1,i_2,j_2\in\mathbb{Z}$, such that $i_1j_2-i_2j_1$ is prime to $m$)
by the matrix
$$[\phi]=[\phi]_{\langle\sigma,\tau\rangle}=\left(
\begin{array}{cc}
i_1 & i_2 \\
j_1 & j_2
\end{array} \right).$$
We show that the symplectic group of $C_m\times C_m$
coincides with
$$\text{SL}_2(m)=\{\phi\in \text{Aut}(C_m\times C_m)| \det[\phi]=1(\text{mod }m)\}.$$
This generalizes the well-known
result when $m$ is prime and $C_m\times C_m$ is viewed as a 2-dimensional vector space over the field of $m$ elements
(see \cite[1.2.6.]{OMeara}).
\begin{lemma}\label{sl2}
Let $[f]\in H^2(C_m\times C_m,\C^*)$. Then identifying Aut$(C_m\times C_m)$ with GL$_2(m)$, we have
\begin{enumerate}
\item SL$_2(m)<$Stab$_{\text{Aut}(C_m\times C_m)}([f])$.
\item If $[f]$ is non-degenerate, then
$\text{Sp}([f])=$SL$_2(m)$.
\end{enumerate}
\end{lemma}
\begin{proof}

Let $C_m\times C_m=\langle \sigma, \tau\rangle$,
and let
\begin{equation}\label{eq:two}
\alpha_f({\sigma},{\tau})=\xi _{m},
\end{equation}
be the alternating form (see \eqref{formab}) associated to $[f]$,
where $\xi _m$ is an $m$-th root of unity.
Let $\phi \in\text{Aut}(C_m\times C_m)$.
A simple calculation shows that
\begin{equation}\label{eq:more}
\alpha_{\phi^{-1}(f)}({\sigma},{\tau})=\alpha_{f}(\phi({\sigma}),\phi(\tau))=
\xi _m ^{\det[\phi]}.
\end{equation}
The cohomology class $[f]$ is $\phi$-stable (or $\phi^{-1}$-stable)
if and only if $\alpha_{\phi(f)}=\alpha_f$.
By~\eqref{eq:two} and~\eqref{eq:more}, the condition $\phi\in$ $\text{SL}_2(m)$, that is det$[\phi]=1(\text{mod }m)$, implies that $\phi$ stabilizes $[f]$. This proves the first part of the claim.

Next, $[f]$ is non-degenerate if and only if the $m$-th root of unity $\xi _m$ is primitive (see \eqref{eq:ndcohonab}). Hence \eqref{eq:two} and \eqref{eq:more} prove that
when $[f]$ is non-degenerate, the condition det$[\phi]=1$ is necessary and sufficient for $\phi$ to stabilize $[f]$. This proves the second part of the claim.
\end{proof}

Using the notion of symplectic actions, we can phrase two results. The first is a closure property of groups of central type to symplectic semidirect products of Hall subgroups.
We are indebted to M. Kochetov for helping us improve the following result.
\begin{theorem}\label{prop:newctsdpr}
Let $G:= N \rtimes _{\eta} K$ be a semidirect product of two finite groups of coprime orders via the action
$$\eta: K\to \text{Aut}(N).$$
Then $G$ is of central type if and only if so are both $N$ and $K$ and, additionally, the action $\eta$ is symplectic.
In particular, if $N=C_m\times C_m$, then $N \rtimes _{\eta} K$ is of central type if and only if $K$ is of central type and
$\eta(K)\in\text{SL}_2(m)$.
\end{theorem}
\begin{proof}
By \eqref{eq:semdircoh}, any cohomology class $[f]\in H^2(G,\C^*)$ is given by a pair $([f_1],[f_2])$, where $[f_1]\in H^2(N,\C^*)^K$ is a $K$-invariant class and
$[f_2]\in H^2(K,\C^*)$.
Since $N$ and $K$ are Hall subgroups of $G$, then by Theorem \ref{th:hall}, non-degeneracy of both $[f_1]$ and $[f_2]$ is equivalent to non-degeneracy of their restrictions to all the Sylow
subgroups of $N$ and $K$ respectively, which in turn is equivalent to the non-degeneracy of $[f]$.
The argument is completed by noting that there exists a class $[f_1]\in H^2(N,\C^*)^K$ which is non-degenerate if and only if $N$ is of central type and the action of $K$ on it is symplectic.
The second part of the claim follows from the first part together with Lemma \ref{sl2}(2).
\end{proof}

The second result concerns groups whose $p$-Sylow subgroups $G_p$ are all abelian. These groups are termed {\it of A-type} and possess nice properties (see \cite{Taunt}).
Let $G$ be of A-type and let $m$ be its order. Applying Theorem \ref{thmswan} for all primes $p|m$ we deduce that any tuple
\begin{equation}\label{Atypecoho}
([f_p])_{p|m}\in \bigoplus_{p|m}H^2(G_p,\mathbb{C}^*)^{W(G_p)}.
\end{equation}
gives rise to a cohomology class $[f]\in H^2(G,\mathbb{C}^*)$ such that res$^G_{G_p}[f]=[f_p]$ (this class is actually unique \cite[Theorem II.6.6]{AM04}).
By Theorem \ref{th:hall} we get
\begin{theorem}\label{thatype}
Let $G$ be a group of A-type.
Then $G$ is of central type if and only if for every prime $p$, $G_p$ is of central type and the action of its Weyl group $W(G_p)$ is symplectic.
\end{theorem}

\subsection{Groups of central type of cube-free order}\label{Proof of Theorem C}
Groups whose orders are cube-free obviously admit abelian $p$-Sylow subgroups and hence are of A-type. Theorem \ref{thatype} described the groups of A-type which are also of central type.
Nonetheless, when their order is cube-free, groups of central type admit a more rigid structure.
In this section we describe all such groups of central type.

The following is a fairly easy claim, which provides an excuse for the focus on cube-free integers.
\begin{lemma}\label{sqrfree}
Suppose that $C_n\times C_n$ is the unique group of central type, up to isomorphism, of order $n^2$. Then $n$ is square-free.
\end{lemma}
\begin{proof}
Assume, in negation, that $n$ contains a square. Let $n=n_1^2\cdot n_2$, where $n_1>1$.
Then the groups $$C_{n_1^2\cdot n_2}\times C_{n_1^2\cdot n_2}$$ and $$(C_{n_1\cdot n_2}\times C_{n_1})\times (C_{n_1\cdot n_2}\times C_{n_1})$$
are non-isomorphic, of order $n^2$ and of central type (see Theorem \ref{abct}).
This contradicts the uniqueness assumption.
\end{proof}
Corollary B gives a necessary and sufficient arithmetic condition for a cube-free number to admit a unique isomorphism type of groups of central type of this order
(see \S\ref{Proof of Corollary B}).

Groups of central type whose orders are a square of {\it one} prime are isomorphic to $C_p\times C_p$ (Theorem \ref{abct}). By Theorem \ref{th:hall}, these are the $p$-Sylow subgroups
of any group of central type of cube-free order.
We move to integers $n$ which are products of two distinct primes $p$ and $q$. Groups of central type of order $p^2q^2$ will play a key part in the proof of Theorem A.
By Lemma \ref{cor:sdpsdp}, we know that groups of such orders are semidirect products of their Sylow subgroups.
We study the corresponding action in the central type case.
\begin{lemma}\label{th:2com}
Let $q\neq p$ be primes and let $G=(C_{q}\times C_{q})\rtimes_{\eta} (C_{p}\times C_{p})$
be a non-abelian group of central type. Then
\begin{enumerate}
\item $G^{\shortmid}=C_{q}\times C_{q}$.
\item The center of $G$ coincides with Ker$(\eta)$, which is cyclic of order $p$.
\end{enumerate}
\end{lemma}
\begin{proof}
Theorem \ref{prop:newctsdpr} tells us that the action $\eta$ is via $\text{SL}_2(q)$.
The two assertions are consequences of the following properties of such an action.

(i) A nontrivial element in $\text{SL}_2(q)$ whose order is prime to $q$ admits no fixed points in the 2-dimensional vector space $\mathbb{F}_{q}^2$.

(ii) For $p=2$ there exists a unique element of order $p$ in $\text{SL}_2(q)$. For odd $p$,
all the $p$-Sylow subgroups of $\text{SL}_2(q)$ are cyclic \cite[\S 7]{MR0074411}.
\end{proof}

The following theorem is a full classification of groups of central type of order $p^2q^2$.
\begin{theorem}\label{th:2pricase}
Let $p<q$ be primes.
\begin{enumerate}
\item If $q \not \equiv \pm 1($mod $p)$, then up to isomorphism there is a unique group of central type of order $p^2q^2$, namely the abelian group
$C_{p\cdot q}\times C_{p\cdot q}.$
\item If $3<q\equiv \pm 1($mod $p)$, then there are exactly two non-isomorphic groups of central type
of order $p^2q^2$, both are
semidirect products
$(C_q\times C_q)\rtimes (C_p\times C_p)$, with SL$_2(q)$-actions.
\item There are exactly three non-isomorphic groups of central type of order $36$ (here $p=2,q=3$), namely, the abelian group of central type
$C_6\times C_6,$
as well as two non-abelian symplectic semidirect products
$(C_3\times C_3)\rtimes (C_2\times C_2)$, and $(C_2\times C_2)\rtimes (C_3\times C_3)$ where the actions are via SL$_2(3)$ and SL$_2(2)$ respectively.
\end{enumerate}
\end{theorem}

\begin{proof}
By Lemma~\ref{cor:sdpsdp}, any group of order $p^2q^2$ is a semidirect product of its Sylow subgroups.
By Theorem \ref{prop:newctsdpr}, $G$ is of central type if and only if
\begin{enumerate}[(i)]
\item $G_p=C_p\times C_p$,
\item $G_q=C_q\times C_q$, and
\item $G_p$ acts on $G_q$ via SL$_2(q)$, or
when $|G|=36$, either $G_2$ acts on $G_3$ via SL$_2(3)$, or $G_3$ acts on $G_2$ via SL$_2(2)$.
\end{enumerate}
Now, since $$|\text{SL}_2(q)|=(q-1)q(q+1),$$ then $q\equiv \pm 1(\text{mod } p)$, that is $p$ divides either $q+1$ or $q-1$, if and only if
SL$_2(q)$ contains an element of order $p$. This is exactly the case where there is a nontrivial symplectic semidirect product $G=(C_q\times C_q)\rtimes (C_p\times C_p)$.
In particular, the abelian group $C_{p\cdot q}\times C_{p\cdot q}$ is the unique group of central type of order $p^2q^2$ (up to isomorphism) if and only if $q$ is not congruent to $\pm 1($mod $p)$,
proving (1).

As mentioned in the proof of Lemma \ref{th:2com}(2), the odd Sylow subgroups of $\text{SL}_2(q)$ are cyclic. This implies that
all the subgroups of order $p>2$ are conjugated in $\text{SL}_2(q)$ as characteristic subgroups in conjugated $p$-Sylow subgroups.
For $p=2$ there is a unique element of order $p$ in $\text{SL}_2(q)$.
Consequently, for either odd or even $p$, any two nontrivial symplectic actions $\eta_1,\eta_2:C_p\times C_p\to$SL$_2(q)$ yield isomorphic semidirect products
$$(C_q\times C_q)\rtimes_{\eta_1} (C_p\times C_p)\cong (C_q\times C_q)\rtimes_{\eta_2} (C_p\times C_p),$$ proving (2).
As mentioned above, in the case where $pq=6$ there is an additional non-abelian group of central type, namely $(C_2\times C_2)\rtimes_{\eta} (C_3\times C_3)$, for a
nontrivial action $\eta:C_3\times C_3\to$SL$_2(2)$.
By similar considerations, this isomorphism type is unique, proving (3).
\end{proof}

We need the following straightforward preliminary result.
\begin{lemma}\label{remark:directct}
Let $G_1$ and $G_2$ be groups of central type, then so is $G=G_1\times G_2$.
\end{lemma}
\begin{proof}
Let $f_i\in Z^2(G_i,\C ^*)$ be non-degenerate 2-cocycles, $i=1,2$.
Then it is easily verified that the 2-cocycle
$$\begin{array}{ccc}
f:G\times G&\to&\C ^*\\
\left((g_1,g_2),(g_1^{\shortmid},g_2^{\shortmid})\right)&\mapsto & f_1(g_1,g_1^{\shortmid})\cdot f_2(g_2,g_2^{\shortmid})
\end{array}$$
is non-degenerate.
\end{proof}

The following lemma will be needed for the proof of Theorem A.
\begin{lemma}\label{linealemma}
Let $V=\mathbb{F}^2$ be a 2-dimensional vector space over a field $\mathbb{F}$ and let $S,T:V\to V$ be two operators such that $T$ is nonsingular and
$S$ is diagonalizable (perhaps over an extension of $\mathbb{F}$).
Suppose that
\begin{equation}\label{comST}
TST^{-1}=S^i
\end{equation}
for some integer $i$.
Then $S$ and $T^2$ commute.
\end{lemma}

\begin{proof}
Working, if necessary, over the algebraic closure $\overline{\mathbb{F}}$, yields a basis of eigenvectors $v,w\in\overline{\mathbb{F}}^2$ of the diagonalizable operator $S$.
If the corresponding eigenvalues $\lambda_v,\lambda_w$ are equal, then evidently the operator $S$ is scalar and commutes with $T$.
We may therefore assume that $\lambda_v\neq\lambda_w$.
The invertibility condition on $T$ implies that $T(v)$ and $T(w)$ are nonzero. Moreover, by \eqref{comST} $T(v)$ and $T(w)$ are eigenvectors of $S^i$.
Both $v$ and $w$ are also eigenvectors of $S^i$, which is conjugate to $S$ and hence admits the distinct eigenvalues $\lambda_v$ and $\lambda_w$.
We deduce that
$$T^2(v)\in\text{span}_{\overline{\mathbb{F}}}(v),\ \ T^2(w)\in\text{span}_{\overline{\mathbb{F}}}(w).$$ Thus, $S$ and $T^2$ commute.
\end{proof}
\begin{subsubsection}{Proof of Theorem A }\label{Proof of Theorem A}

The ``if" argument follows immediately from Theorem \ref{prop:newctsdpr}.
We prove the ``only if" direction in a way of induction on $r$, that if $G$ is a group of central type of order $n^2=\prod _{i=1}^rp_i^{2}$ (where $\{p_i\}_{i=1}^r$ are distinct primes),
then

(*) {\it $G^{\shortmid}$ is the direct product of all its Sylow subgroups $G_{p_i}$, such that $G_{p_i}=(G_{p_ip_j})^{\shortmid}$ for some $1\le j \le r$.}

By Lemma \ref{th:2com}(1), the condition (*) implies that
$G^{\shortmid}$ is an abelian Hall subgroup of $G$.
If $G^{\shortmid}$ is indeed an abelian Hall subgroup of $G$, then by Theorem \ref{sz}, $G= G^{\shortmid}\rtimes_{\eta} G_0$, where $G_0$ is a Hall subgroup of $G$.
Since Hall subgroups of $G$ are also of central type (Theorem \ref{th:hall}), and since the Hall subgroups $G^{\shortmid}$ and $G_0$ are abelian
and of cube-free orders, then by Theorem \ref{abct}
they are isomorphic to $C_m\times C_m$ and $C_k\times C_k$ for some coprime numbers $m$ and $k$ the product of which is $n$. Furthermore, by Theorem \ref{prop:newctsdpr}, the action $\eta$
of $G_0$ on $G^{\shortmid}$ is via SL$_2(m)$. This will prove the theorem.

Obviously, (*) holds for $r=1,2$.
Let $G$ be a group of central type of cube-free order.
The hypothesis enables us to assume that all the proper subgroups of $G$ that are of central type (these are exactly its Hall subgroups) satisfy (*).
By Lemma \ref{cor:semi}(2) for the case where 6 does not divide the order of $G$, or by Lemma \ref{prop:6primehall} for the case where 6 does divide $|G|$, there exists a prime $p$ which divides
$|G|$ such that the $p^{\shortmid}$-Hall subgroup $G_{p^{\shortmid}}$ is a proper normal subgroup of $G$. We have
\begin{equation}\label{p'p}
G=G_{p^{\shortmid}}\rtimes G_p\cong G_{p^{\shortmid}}\rtimes (C_p\times C_p).
\end{equation}
Next, Theorem \ref{th:hall} tells us that $G_{p^{\shortmid}}$ is of central type.
By the induction hypothesis,
$(G_{p^{\shortmid}})^{\shortmid}$ satisfies (*). In particular it is a Hall subgroup of $G_{p^{\shortmid}}$. Being characteristic in $G_{p^{\shortmid}}$,
it a normal Hall subgroup of $G$. Again, by Theorem \ref{sz}, there exists a Hall subgroup (not necessarily proper) $G_k<G$ such that
\begin{equation}\label{pk}
G=(G_{p^{\shortmid}})^{\shortmid}\rtimes_{\nu} G_k,
\end{equation}
where the action $\nu$ is symplectic.
If $(G_{p^{\shortmid}})^{\shortmid}$ is trivial, then \eqref{p'p} certainly describes $G$ as a semidirect product of abelian Hall subgroups of central type.
Applying Lemma \ref{th:2com}(1) it is not hard to check that in this case $G^{\shortmid}$ is the direct product of the $q$-Sylow subgroups
$G_q$ such that the $pq$-Hall subgroups $G_{pq}$ are non-abelian. This is just to say that $G$ satisfies (*).

We may hence assume that $(G_{p^{\shortmid}})^{\shortmid}$ is nontrivial,
that is $G_k$ is a proper subgroup of $G$. Since $G_k$ is of central type, it satisfies the induction hypothesis (*).
In particular, the commutator $(G_k)^{\shortmid}$ is a Hall subgroup. Theorem \ref{sz} ascertains again the existence of another Hall subgroup $G_l$ such that
$$G_k=(G_k)^{\shortmid}\rtimes G_l.$$
Mutual inclusion verifies that \eqref{pk} yields
\begin{equation}\label{G'pk}
G^{\shortmid}=(G_{p^{\shortmid}})^{\shortmid}\rtimes_{\nu} (G_k)^{\shortmid}
\end{equation}
(we keep the notation $\nu$ also for its restriction to $(G_k)^{\shortmid}$).
By the induction hypothesis, both $(G_{p^{\shortmid}})^{\shortmid}$ and $(G_k)^{\shortmid}$ are direct sums of commutator subgroups of the form $(G_{p_ip_j})^{\shortmid}$.
Therefore, to complete the proof, we need to show that the action in \eqref{G'pk} is trivial. These groups are abelian and hence admit prime decomposition. Therefore, in order
to establish triviality of this action, it is enough to show that under the symplectic action $\nu$
every Sylow subgroup $G_{q_1}<(G_k)^{\shortmid}$ centralizes every Sylow subgroup $G_{q_2}<(G_{p^{\shortmid}})^{\shortmid}$.
Assume that this is not true. Namely, there exist prime divisors $q_1$ and $q_2$ of $|(G_k)^{\shortmid}|$ and $|(G_{p^{\shortmid}})^{\shortmid}|$ respectively
such that $G_{q_1}$ acts nontrivially on $G_{q_2}$. By Lemma \ref{th:2com}(2), the center of $G_{q_1\cdot q_2}$ is cyclic of order $q_1$.

Recall that by the induction hypothesis on $G_k$, there exists a prime $q_3$ (dividing $l$) such that
$G_{q_1}=(G_{q_1\cdot q_3})^{\shortmid}$ for every Hall subgroup $G_{q_1\cdot q_3}<G_k$. In particular, there is an element $t\in G_{q_3}(<G_l)$
conjugating the elements of $G_{q_1}$ nontrivially (by a symplectic action, as $G_{q_1\cdot q_3}=\langle G_{q_1},G_{q_3}\rangle$ is of central type).
Since this nontrivial action has a cyclic kernel of order $q_3$, then
by Lemma \ref{th:2com}(1) we obtain
\begin{equation}\label{[]}
G_{q_1}=[G_{q_1},t].
\end{equation}
Let $\nu_{q_2}:G_{q_1\cdot q_3}\to \text{SL}_2(q_2)$ be the restriction of the action $\nu$ (see \eqref{pk}) from the abelian group $(G_{p^{\shortmid}})^{\shortmid}$ to
its $q_2$-Sylow subgroup $G_{q_2}$.
Our goal is to show that $\nu_{q_2}(G_{q_1})$ commutes with $\nu_{q_2}(t).$ Using \eqref{[]}, this will show that
$$\nu_{q_2}(G_{q_1})=\nu_{q_2}([G_{q_1},t])=[\nu_{q_2}(G_{q_1}),\nu_{q_2}(t)]$$
is trivial, contradicting our assumption that $G_{q_1}$ acts nontrivially on $G_{q_2}$.

Consider the action of $t$ on $G_{q_1\cdot q_2}=\langle G_{q_1},G_{q_2}\rangle$.
Let $$s_1\in{\text Z}(G_{q_1\cdot q_2})$$
be any non-trivial element in the cyclic center of the Hall subgroup $G_{q_1\cdot q_2}$. Since this center is characteristic, there exists an integer $i_1$ such that
\begin{equation}\label{ts1}
ts_1t^{-1}=s_1^{i_1}.
\end{equation}
Furthermore, since $t$ acts on $G_{q_1}$ via SL$_2(q_1)$ with an eigenvector $s_1\in C_{q_1}\times C_{q_1}$ and since its order $q_3$ is prime to $q_1$, this action is diagonalizable.
Hence, there is another eigenvector
$s_2\in C_{q_1}\times C_{q_1}$ and another integer $i_2$ such that
\begin{equation}\label{ts2}
ts_2t^{-1}=s_2^{i_2}.
\end{equation}
Let
$$T=\nu_{q_2}(t), S_1=\nu_{q_2}(s_1), S_2=\nu_{q_2}(s_2).$$
Clearly, since the elements $s_1$ and $s_2$ generate the group $ C_{q_1}\times C_{q_1}$, then the operators $S_1$ and $S_2$ generate $\nu_{q_2}(G_{q_1})$.
It is then enough to show that these operators commute with $T$. Suppose first that $q_1=2$ (and then $q_2>2$). Then $T$ is the unique operator of order 2 in SL$_2(q_2)$, namely the one taking
any element to its inverse, which is central in SL$_2(q_2)$. We may hence assume that the prime $q_1$ is odd.
Now, the operators $S_1$ and $S_2$ are in SL$_2(q_2)$ and of order $q_1$, which is prime to $q_2$. Again, this is enough to ensure that they are diagonalizable.
Moreover, \eqref{ts1} and \eqref{ts2} yield
$$TS_jT^{-1}=S_j^{i_j}, \ \ j=1,2.$$
This permits us to apply Lemma \ref{linealemma} (with $\mathbb{F}=\mathbb{F}_{q_2}$) so as to establish that both $S_1$ and $S_2$ commute with $T^2$.
Since the order of $T$ is $q_1$, which is prime to 2, then $T$ is a power of $T^2$ and so $S_1$ and $S_2$ commute with $T$ itself. This completes the proof. \qed

\end{subsubsection}

\subsubsection{Proof of Corollary B}\label{Proof of Corollary B}
If $\{p_i\}_{i\in \lambda}$ is a set of distinct primes, then clearly
\begin{equation}\label{sl2m}
|\text{SL}_2(\prod_{i\in \lambda}p_i)|=\prod_{i\in \lambda}(p_i-1)p_i(p_i+1).
\end{equation}
Corollary B is a consequence of Theorem A, Lemma \ref{sqrfree} and \eqref{sl2m}.
\qed

\begin{remark}\label{density}
The family of square-free numbers is of density $\frac{6}{\pi ^2}$ in the natural numbers (see \cite[Theorem 2.1]{Sinai}).
The family of numbers $\Lambda$ is of density 0 in $\mathbb{N}$. However, using the Chinese Reminder Theorem and Dirichlet's theorem, one can show that this family is
big in the sense that for any $n\in \Lambda$ there exist infinitely many primes $p$ such that $np\in \Lambda$.
\end{remark}

\subsection{The automorphism action on non-degenerate classes is not necessarily transitive}\label{non-transitive}
The following examples are of groups $G$ of central type such that
the action of Aut$(G)$ on the non-degenerate cohomology classes in
$H^2(G,\mathbb{C}^*)$ admits more than one orbit.
\begin{example}\label{example0}
The action on the non-degenerate classes may be non-transitive even when the corresponding groups of central type are of cube-free order.
Theorem A says that any such group $G$ is a semidirect product
$$(C_m\times C_m)\rtimes_{\eta} (\mathop{C_k}^{\langle x\rangle} \times \mathop{C_k}^{\langle y\rangle}),$$
where $(m,k)=1$ and where the action $\eta$ is via SL$_2(m)$. By Lemma \ref{sl2}(1), all the classes in $H^2(C_m\times C_m,\C^*)$ are stabilized by $\eta$,
and so by \eqref{eq:semdircoh} we have
$$H^2(G,\C^*)=H^2(C_m\times C_m,\C^*)\times H^2(C_k\times C_k,\C^*)(\cong C_{mk}).$$
Any cohomology class $[f]\in H^2(G,\C^*)$ can be decomposed to its $m$ and $k$-parts and so can be given as a pair $(\alpha_{[f_1]},\alpha_{[f_2]})$ of the alternating forms (see \eqref{altform})
associated to its restrictions to the corresponding
subgroups $[f_1]\in H^2(C_m\times C_m,\C^*)$ and $[f_2]\in H^2(C_k\times C_k,\C^*)$.
The following result can be proven using the fact that $m$ and $k$ are coprime and is left to the reader:
Let $n_1,n_2\in C_m\times C_m$ and $h_1,h_2\in C_k\times C_k$ such that $n_1h_1$ and $n_2h_2$ commute (that is $(n_1h_1,n_2h_2)\in\mathcal{A}_G$, see \S\ref{perct}). Then
for every $[f]\in H^2(G,\C^*)$
\begin{equation}\label{alpha2}
\alpha_{[f]}(n_1h_1,n_2h_2)=\alpha_{[f]}(h_1,h_2).
\end{equation}

Let $k$ be an odd prime and let $m=p_1\cdot p_2\cdot p_3$, where $p_1,p_2,p_3$ are distinct primes that are congruent to $\pm 1($mod $k)$.
Then there are nontrivial actions $$\eta_i:\langle x,y\rangle\to \text{SL}_2(p_i),\ \ i=1,2,3,$$
where the kernels of $\eta_1, \eta_2$ and $\eta_3$ are $\langle x\rangle , \langle y\rangle $ and $\langle xy\rangle $ respectively. Let
$$\eta:\langle x,y\rangle\to \text{SL}_2(m)=\prod_{i=1}^{3}\text{SL}_2(p_i)$$
be the combined (faithful) action (see \eqref{sl2m}).

Let $g\in$ ker$(\eta_i)<C_k\times C_k$. Then its image under any automorphism of $G$ should obviously be inside $(C_m\times C_m)\rtimes$ ker$(\eta_i)$.
In particular, for any $\psi\in$Aut$(G)$ we have
\begin{eqnarray}\label{psi}
\begin{array}{ll}
\psi(x)\in(C_m\times C_m)\rtimes\langle x\rangle,& \\ \psi(y)\in(C_m\times C_m)\rtimes\langle y\rangle,& \text{and}\\ \psi(xy)\in(C_m\times C_m)\rtimes\langle xy\rangle.&
\end{array}
\end{eqnarray}
Using \eqref{psi}, it is not hard to check that there exists $1\leq j\leq k-1$ depending only on $\psi$, such that every $g\in C_k\times C_k$
is mapped to $\psi(g)=n_gg^j$ for some $n_g\in C_m\times C_m$.
Let $[f]\in H^2(G,\C^*)$ be given by the pair $(\alpha_{[f_1]},\alpha_{[f_2]})$ of alternating forms associated to the $m$ and $k$-parts of $[f]$.
Then by the above, there exist $1\leq j\leq k-1$ and $n_x,n_y\in C_m\times C_m$ such that
$$\alpha_{\psi([f])}(x,y)=\alpha_{[f]}(\psi^{-1}(x),\psi^{-1}(y))=\alpha_{[f]}(n_xx^j,n_yy^j).$$
Applying \eqref{alpha2} we get
$$\alpha_{\psi([f])}(x,y)=\alpha_{[f]}(x^j,y^j)=\alpha_{[f]}(x,y)^{j^2}.$$
Now, by Lemma \ref{altformorder}, $\alpha_{[f]}(x,y)$ is a $k$-th root of 1.
For any such root $\xi_k$, the set $\{\xi_k^{j^2}\}^{k-1}_{j=1}$ does not cover the entire set of
$k$-th roots of 1 (it contains exactly half of them since $\alpha_{[f]}(x,y)$ is primitive). It follows that there is a non-degenerate class in $H^2(G,\C^*)$ which is not in the Aut$(G)$-orbit of $[f]$.
For example, take $k=3$. Then, for any $\psi\in$Aut$(G)$, the above $j$ is either 1 or 2. In both cases $\xi_3^{j^2}=\xi_3$ and so if $[f],[f']\in H^2(G, \C^*)$ satisfy
$\alpha_{[f]}(x,y)=\xi_3$ and $\alpha_{[f']}(x,y)=\xi_3^2,$ then they do not belong to the same Aut$(G)$-orbit.\qed
\end{example}

Notice that although the above cohomology classes are not in the same Aut$(G)$-orbit, one cannot distinguish between them in practice.
These classes do belong to the same orbit under the following action of Aut$(G)\times$Aut$(\C^*)$ on the (second) cohomology of any group which generalizes \eqref{eq:action}.
Let $[f]\in H^2(G,\C^*)$ and let $(\phi,\rho)\in$Aut$(G)\times$Aut$(\C^*)$.
Then
\begin{equation}\label{genaction}
\begin{aligned}
(\phi,\rho)([f])&=[{(\phi,\rho)}(f)],\ \ \text{where} \\
{(\phi,\rho)}(f)(g_1,g_2):&=\rho(f({\phi}^{-1}(g_1),{\phi}^{-1}(g_2))).
\end{aligned}
\end{equation}
Certainly, the classes $[f],[f']\in H^2(G, \C^*)$ in Example \ref{example0} are in the same orbit under the action \eqref{genaction}.
This is not the case in the next two examples, where non-degenerate classes of the same group
may behave totally different.

\begin{example}\label{example1}
This example is of a family of
groups $G$ of A-type, i.e. which admit only abelian $p$-Sylow subgroups, that are also of central type. In
particular, the action of Aut$(G_p)$ on the non-degenerate
cohomology classes in $H^2(G_p,\mathbb{C}^*)$ is transitive for every $p$-Sylow subgoup $G_p<G$ (see \S\ref{intro}).

Let $q_1, q_2,q_3$ be distinct\footnote{The example holds also when $q_1=q_2$, the argument should be modified accordingly}
primes such that $q_i \equiv 1($mod $q_3),\ \ i=1,2$, and let
$$G_i=\left(\mathop{C_{q_i}}^{\langle x_i\rangle}\times
\mathop{C_{q_i}}^{\langle y_i \rangle}\right)\rtimes_{\eta_i} \left(\mathop{C_{q_3}}^{\langle
\sigma _i \rangle}\times \mathop{C_{q_3}}^{\langle \tau _i \rangle}\right),\quad i=1,2.$$
The actions are given by
\begin{eqnarray}
\begin{array}{rcl}
\eta_i:\langle \sigma_i, \tau_i\rangle &\to &\text{Aut}(\langle x_i, y_i\rangle)\\
\eta_i(\tau_i):& &x_i\mapsto x_i,\\
& &y_i\mapsto y_i,\\
\eta_i(\sigma_i):& & x_i\mapsto x_i^{m_i},\\
& &y_i\mapsto y_i^{n_i},
\end{array}
\end{eqnarray} where for $i=1,2$, $m_i,n_i\in\mathbb{Z}$ satisfy
\begin{equation}\label{omzet}
m_i\cdot n_i\equiv1(\text{mod } q_i)
\end{equation}
and
$$m_i^{q_3}\equiv1(\text{mod } q_i),\quad m_i\not\equiv 1(\text{mod } q_i).$$
Condition \eqref{omzet} says that for $i=1,2$, the actions $\eta_i$ take their images in SL$_2(q_i)$ (identified as before as a subgroup of $\text{Aut}(\langle x_i, y_i\rangle)$),
and therefore by Lemma \ref{sl2}(2) both are symplectic.
By Theorem~\ref{prop:newctsdpr}, the groups $G_1$ and $G_2$ are of central type and by Lemma~\ref{remark:directct} so is
their direct product
$$G:=G_1\times G_2.$$
Since $m_i$ and $n_i$ are not congruent to 1(mod $ q_i$), there are no non-trivial
elements in $C_{q_i}\times C_{q_i}$
which are fixed under the action $\eta_i$. Hence, the center of $G$ is
\begin{equation}\label{centerexample}
Z(G)=\langle \tau_1,\tau _2 \rangle.
\end{equation}
We can write the group $G$ as a semidirect product of its Hall subgroups
$$G=G_{q_1\cdot q_2}\rtimes_{\eta} G_{q_3},$$
where
$$G_{q_3}=\langle \sigma _1, \sigma _2, \tau _1, \tau_2\rangle\cong C_{q_3}^4,\ \ G_{q_1\cdot q_2}=
\langle x_1\cdot x_2,y_1\cdot y_2\rangle\cong C_{q_1\cdot q_2}\times C_{q_1\cdot q_2},$$
and the action $\eta:G_{q_3}\to$SL$_2(q_1\cdot q_2)$ is determined by the actions $\eta_1$ and $\eta_2$ in an obvious way.

Let $[f_1],[f_2]\in H^2(G_{q_3},\C^*)$ be two cohomology classes determined
by their associated $G_{q_3}$-form (see \eqref{altform}) on the generators $\sigma_1,\tau_1,\sigma_2,\tau_2$.
\begin{equation}\label{f1f2}
\begin{array}{ll}
\alpha_{f_1}:(\tau _1,\sigma _1),(\tau _2,\sigma _2)\mapsto\xi _{q_3},&
(\tau _1,\sigma _2),(\tau _2,\sigma_1),(\tau _1,\tau _2),(\sigma _1,\sigma_2)\mapsto 1,\\
\alpha_{f_2}:(\tau _1,\tau _2),(\sigma _1,\sigma _2)\mapsto\xi _{q_3},&
(\tau _i,\sigma _j)\mapsto 1, \ \ 1\leq i,j\leq 2.
\end{array}
\end{equation}
Here, $\xi _{q_3}$ is a primitive
${q_3}$-\textit{th} root of unity. By~\eqref{eq:ndcohonab}, the two
cohomology classes $[f_1]$ and $[f_2]$ are non-degenerate.

Next, since the action $\eta$ of $G_{q_3}$ on $G_{q_1\cdot q_2}$ is via SL$_2(q_1\cdot q_2)$, then by Lemma \ref{sl2}(1) we have
\begin{equation}\label{slactriv}
H^2(G_{q_1\cdot q_2},\C^*)^{G_{q_3}}=H^2(G_{q_1\cdot q_2},\C^*).
\end{equation}
By \eqref{eq:semdircoh} and \eqref{slactriv} we get
$$H^2(G,\mathbb{C}^*)=H^2(G_{q_1\cdot q_2},\C^*)\times H^2(G_{q_3},\C^*).$$
Thus, every cohomology class of $G$ is given by a pair of classes $[f]\in H^2(G_{q_1\cdot q_2},\C^*)$ and
$[f']\in H^2(G_{q_3},\C^*)$.
By Theorem \ref{th:hall}, $([f],[f'])\in H^2(G,\C^*)$ is non-degenerate if and only if so are both classes $[f]$ and $[f']$.

Now, the group $G_{q_1\cdot q_2}$ is of central type and so admits non-degenerate cohomology classes.
For every such non-degenerate class $[f]\in H^2(G_{q_1\cdot q_2},\C^*)$, the cohomology classes
$$([f],[f_1]), ([f],[f_2])\in H^2(G,\C^*)$$
(where the classes $[f_1],[f_2]$ are given in \eqref{f1f2})
are non-degenerate. We claim that these two classes lie in different Aut$(G)$-orbits.
Indeed, examine the restriction of the symplectic forms associated to both classes to the center
$\langle \tau_1,\tau _2 \rangle$ of $G$ (see \eqref{centerexample}). We have
$\alpha_{([f],[f_1])}(\tau_1,\tau_2)=1$, whereas $\alpha_{([f],[f_2])}(\tau_1,\tau_2)=\xi _{q_3}$.
That is, on this subgroup one form is trivial while the other is not (it is even non-degenerate).
Since the center is characteristic, there exists no
$G$-automorphism taking one of the above classes to the other.\qed
\end{example}

\begin{example}\label{example2}
Next, we present another, essentially different example of a
$2$-group $G$ of central type such that the action of Aut$(G)$ on
the non-degenerate cohomology classes in $H^2(G,\mathbb{C}^*)$ is
not transitive.
The group $G=\check{A}\rtimes Q$ is a semidirect product of a group $Q$ acting on the character group $\check{A}$
of a $Q$-module $A$. Here $\check{A}:=$Hom$(A,\C^*)$ is endowed with the diagonal $Q$-action
\begin{equation}\label{diag}
\langle q(\chi),a\rangle:=\langle\chi,q^{-1}(a)\rangle,
\end{equation}
for every $q\in Q, \chi\in \check{A}$, and evaluated at any $a\in A$.
We apply a technique developed in \cite{eg3} to construct a non-degenerate 2-cocycle on the
above semidirect product from a bijective 1-cocycle $\pi:Q\to A$.
Let $Q_0:=C_4=\langle x\rangle$ act on the Klein 4-group
$A_0:=C_2\times C_2=\langle\sigma\rangle\times\langle\tau\rangle$ as follows. The generator
$x$ interchanges the generators $\sigma$ and $\tau$ of $A_0$. Then it is not hard to check that the
following correspondence is a bijective 1-cocycle
$$\begin{array}{c}
\pi_0:Q_0\to A_0\\
1\mapsto 1,\ \  x\mapsto \sigma, \ \ x^2 \mapsto \sigma\tau,\ \
x^3\mapsto \tau.
\end{array}$$
Next, we extend the 1-cocycle $\pi_0$ to
$$\begin{array}{rcl}
\pi:Q_0\times Q_0&\to &A_0\times A_0\\
(q_1,q_2)&\mapsto &(\pi_0(q_1),\pi_0(q_2)),
\end{array}$$
where $A:=A_0\times A_0$ is endowed with the componentwise $Q:=Q_0\times
Q_0$-action, inducing the diagonal $Q$-action \eqref{diag} on $\check{A}$. Then $\pi$ is a bijective 1-cocycle, such that
$[\pi]\in H^1(Q,A)$ is of order 2.
Let $$G:=\check{A}\rtimes Q=(\check{A_0}\times \check{A_0})\rtimes (Q_0\times Q_0).$$
Using the terminology of \cite[(3.6)]{bg}, the class
$$[c_{\pi}]\in H^2(G,\mathbb{C}^*)$$ is non-degenerate.
Note that
\begin{equation}\label{rescpi}
\text{res}^G_{Q}[c_{\pi}]=1,
\end{equation}
and that the
order of $[c_{\pi}]$ is 2, since the values of $c_{\pi}$ are either
1 or $-1$. On the other hand, let $[c]\in H^2(Q,\mathbb{C}^*)$ be the cohomology class determined by its associated $Q$-form (see \eqref{altform})
on the generators $(x,1)$ and $(1,x)$
$$\alpha_c((x,1),(1,x))=i.$$
The class $[c]$ is of order 4 (the fact that it is also non-degenerate
is irrelevant). By the splitting of $Q$ in $G$ we have
\begin{equation}\label{resinf}
\text{res}^G_Q(\text{inf}^{Q}_G[c])=[c].
\end{equation}
Hence, inf$^Q_G[c]$ is also of order 4 in $H^2(G,\mathbb{C}^*)$.

By \cite[\S{4}]{bg}, the class
$$[\beta]:=[c_{\pi}]\cdot\text{inf}^{Q}_G[c]\in H^2(G,\mathbb{C}^*)$$
is non-degenerate. Moreover, by \eqref{rescpi} and \eqref{resinf}
$$\text{res}^G_{Q}[\beta]=
\text{res}^G_{Q}[c_{\pi}]\cdot\text{res}^G_{Q}(\text{inf}^{Q}_G[c]) =1\cdot[c]=[c]$$ is of order 4 in $H^2(Q,\C^*)$,
then so is the order of $[\beta]$ in $H^2(G,\mathbb{C}^*)$.

To sum up, $H^2(G,\mathbb{C}^*)$ admits two non-degenerate
classes, $[c_{\pi}]$ and $[\beta]$, which cannot be in the same
orbit under the Aut$(G)$-action for they are of distinct orders- 2
and 4 respectively.\qed
\end{example}

\section {Maximal connected gradings of semisimple algebras}
Our objective in this section is to study the maximal connected grading classes, first of finite-dimensional simple complex algebras (\S\ref{mcsg}),
and then of general finite-dimensional semisimple complex algebras (\S \ref{proofD}). Special semisimple families, namely diagonal, bisimple and completely graded decomposition
algebras are discussed in \S\ref{diagsec}, \S\ref{bisimpsec} and \S\ref{cgda} respectively.

By Theorem \ref{prop:decsemsi}, any grading of such an algebra $A$ admits a decomposition as a direct sum of graded-simple algebras.
When the grading of $A$ is connected, we may think of all the graded-simple summands as connected by appropriate well-defined subgroups (see \S\ref{cps}). The following claim is straightforward.
\begin{lemma}\label{decompequiv}
Let \eqref{eq:equivgr} be connected gradings of finite-dimensional semisimple algebras $A$ and $B$.
Then any graded equivalence $(\psi,\phi)$ between these gradings determines a bijection between the graded-simple summands $\{A_i\}_{i=1}^r$ of $A$ and $\{B_j\}_{j=1}^s$ of $B$ (in particular $r=s$),
as well as graded-equivalences $(\psi_i,\phi_i)$ between the connected-gradings of the summands $A_i$ and $\psi(A_i)$ for every $1\leq i\leq r$.
\end{lemma}

The next result is somehow counterintuitive.
Let $(\psi,\phi)$ be a quotient graded-morphism from a graded algebra $A=\bigoplus _{g\in G} A_g$ to a graded algebra $B=\bigoplus _{h\in H} B_h$. If $I$ is a non-trivial graded ideal of $A$
then $\psi(I)$ is a non-trivial graded ideal of $B$.
Consequently, we have
\begin{lemma}\label{simplequotient}
Let $\mathcal{G}_B:B=\bigoplus _{h\in H} B_h$ be a quotient-grading of $\mathcal{G}_A:A=\bigoplus_{g\in G} A_g$. If the $H$-grading of $B$ is simple, then so is the $G$-grading of $A$.
\end{lemma}

\subsection{Maximal connected simple gradings}\label{mcsg}
{With the notation of Theorem \ref{th:BSZ}, let $m=\epsilon(x)$ and let \eqref{AW} be the decomposition of the twisted group algebra $\C^fG$ to its simple summands,
each of which corresponds to an irreducible $f$-projective representation of the subgroup $G$ of $\Gamma$. The following claim is clear.
\begin{lemma}\label{indform}
With the above notation, $x\left(\C^fG\right)$ is a $\Gamma$-simple grading class
of the semisimple algebra
\begin{equation}\label{Adec}
A=\bigoplus_{[W]\in\text{Irr}(G,f)}M_{n(W)}(\C),
\end{equation}
where $n(W)=m\cdot\dim_{\C}(W)$.
\end{lemma}
We now specify a group $\Gamma$ as well as a $\Gamma$-simple grading class $x\left(\C^fG\right)$ of the semisimple algebra $A$ given in \eqref{Adec}.
Fix any set of generators $\{x_1,\cdots,x_{{m}-1}\}$ of the free group $\mathcal{F}_{{m}-1}$ of rank $m-1$.
Let
\begin{equation}\label{F*G}
\Gamma=\mathcal{F}_{m-1}*G
\end{equation}
be the free product of $G$ and $\mathcal{F}_{{m}-1}$, and let $$\tilde{x}=1+\sum_{i=1}^{{m}-1}x_i\in\N\Gamma.$$
Obviously, the $\Gamma$-simple grading class $\tilde{x}\left(\C^fG\right)$ of $A$ is connected. In fact, we claim that it is maximal connected.
To prove it we need

\begin{lemma}\label{hatvstilde}
With the above notation, if $\hat{x}(\C^fG)$ is the same grading class as $\tilde{x}(\C^fG)$, then $(\tilde{x},\hat{x})\in R_G$.
In particular, $\hat{x}=g+\sum_{j=1}^{m-1}\hat{x}_j$,
where $g\in G$, and $\hat{x}_j$ are torsion-free for every $1\leq j\leq m-1$.
\end{lemma}
\begin{proof}
The lemma follows from the condition for graded-equivalence in Theorem \ref{th:AHequi}, bearing in mind that the normalizer of $G$ in \eqref{F*G} is $G$ itself.
\end{proof}
We can now claim
\begin{lemma}\label{xcfg}(see \cite[Proposition 2.31]{EK13})
With the above notation, the $\Gamma$-simple grading class
\begin{equation}\label{scgr}
(1+\sum_{i=1}^{m-1}x_i)(\C^fG)
\end{equation} is maximal.
\end{lemma}
\begin{proof}
Suppose that the class \eqref{scgr} is a quotient of some $\hat{\Gamma}$-grading class $[\mathcal{G}]$. By Lemma \ref{simplequotient}, the grading class $[\mathcal{G}]$ is simple and hence
is of the form $\hat{y}(\C^{\hat{f}}\hat{G})$.
Let $$\phi:\hat{\Gamma}\to\mathcal{F}_{m-1}*G$$
be a surjective group homomorphism realizing the quotient morphism.
Since graded morphisms respect induction (see Lemma \ref{lemma:quoinduce}),
the graded subalgebra (see \S\ref{inductionsec}) $\C^{\hat{f}}\hat{G}$ is mapped to a class $x(\C^{f}G)$ for some $x\in \N\phi(\hat{G})$ such that \eqref{scgr} is the same grading class
as $\bar{\phi}(\hat{y})\cdot x(\C^{f}G)$ (with the extension notation \eqref{extended}).
Now, by the normalization as described in Remark \ref{1inx} we may assume that $\hat{y}=1+\sum_{j=1}^{\epsilon(\hat{y})-1}\hat{y}_j.$
Consequently, \begin{equation}\label{barhat}
\bar{\phi}(\hat{y})\cdot x=(1+\sum_{j=1}^{\epsilon(\hat{y})-1}\phi(\hat{y}_j))\cdot x=x+(\sum_{j=1}^{\epsilon(\hat{y})-1}\phi(\hat{y}_j))\cdot x.
\end{equation}
Next, $x\in \N\phi(\hat{G})$, hence its support consists of elements of finite orders. Therefore, by \eqref{barhat}, the class $\bar{\phi}(\hat{y})\cdot x(\C^{f}G)$ is induced
from the twisted group algebra $\C^fG$ by an element $x+(\sum_{j=1}^{\epsilon(\hat{y})-1}\phi(\hat{y}_j))\cdot x$, which consists at least $\epsilon(x)$-many group elements of finite order.
Since it represents the same grading class as \eqref{scgr}, Lemma \ref{hatvstilde} implies that $\epsilon(x)=1$, moreover, $x$ is an element of $G$.
The fact that the graded subalgebra $\C^{\hat{f}}\hat{G}$ is mapped to a class $x(\C^{f}G)$ with $\epsilon(x)=1$ says that
the restriction $\phi|_{\hat{G}}$ to the subgroup $\hat{G}<\hat{\Gamma}$ is an isomorphism and $\phi([\hat{f}])=[f]$. Thus,
both grading classes are induced from the same twisted group algebra grading class - one by $\hat{y}$ and
the other by $\tilde{x}=1+\sum_{i=1}^{m-1}x_i$.
Comparing dimensions of the inducing graded vector spaces we now get
$$\epsilon(\hat{y})=\epsilon(\tilde{x})=m.$$
Furthermore, Lemma \ref{hatvstilde} also implies that $(\bar{\phi}(\hat{y})\cdot x,\tilde{x})\in R_G.$ Since $x\in G$ we have
$$(\bar{\phi}(\hat{y}),\tilde{x})\in R_G.$$
That is, there exist $g_1,\cdots, g_{m-1}\in G$ such that
$$\sum_{j=1}^{m-1}\phi(\hat{y}_j)g_j=\sum_{i=1}^{m-1}x_i.$$
For every $1\leq j\leq m-1$ let $\hat{g}_j$ be a pre-image of $g_j$ in $\hat{G}$.
Then $$\{\hat{y}_j\hat{g}_j\}_{j=1}^{m-1}$$ is a set of pre-images of the set $\{x_i\}_{i=1}^{m-1}$ of generators of $\mathcal{F}_{m-1}$ whose cardinality is also $m-1$.
Clearly,
$$\langle\hat{y}_j\hat{g}_j\rangle_{j=1}^{m-1}\cong \mathcal{F}_{m-1}.$$
Furthermore, since $\Gamma$ is a free product of $G$ and a free group of rank $m-1$ generated by $\{x_i\}_{i=1}^{m-1}$,
then $$\hat{\Gamma}=\langle \hat{G},\hat{y}_1\hat{g_1},\cdots,\hat{y}_{m-1}\hat{g}_{m-1}\rangle\cong \mathcal{F}_{m-1}*G.$$
It is then easy to deduce that $\phi$ is a group isomorphism and so any quotient morphism to \eqref{scgr} is actually a graded equivalence. In other words, \eqref{scgr} is maximal connected.
\end{proof}
The grading classes of the form \eqref{scgr} are the only maximal graded-simple classes of a finite-dimensional complex semisimple algebra. This is a consequence of the following claim,
whose proof is given in the more general setup of Lemma \ref{scgr2} later on.
\begin{lemma}\label{scgr1}(see \cite[Corollary 2.34]{EK13})
Let $\mathcal{G}$ be a connected $\Gamma_1$-simple grading of a finite-dimensional complex semisimple algebra $A$, where $\Gamma_1$ is any group. Then, with the above notation
$\mathcal{G}$ is a quotient of a grading in some class \eqref{scgr}.
\end{lemma}

At this point, the one-to-one correspondence in Theorem \ref{A} is straightforward:\\
{\it Proof of Theorem \ref{A}.}
Let $(G,\gamma)\in \mathcal{P}_n$, where $G$ is a group of central type of order $d^2$ (dividing $n^2$) and $\gamma$
is an Aut$(G)$-orbit of non-degenerate classes. This pair corresponds to the $\mathcal{F}_{\frac{n}{d}-1}*G$-grading class \eqref{scgr} of $M_n(\C)$
with $m=\frac{n}{d}$ and $[f]\in H^2(G,\C^*)$ in the orbit $\gamma$.
Firstly, it is not hard to deduce from Theorem \ref{th:AHequi} that distinct pairs in $\mathcal{P}_n$ correspond to distinct grading classes.
Secondly, Lemma \ref{xcfg} says that the grading classes \eqref{scgr} are indeed maximal. Finally, by Lemma \ref{scgr1}
any connected grading class of $M_n(\C)$ is a quotient of at least one grading class that corresponds to a pair in $\mathcal{P}_n$.
\qed

The following two observations are direct consequences of the definition of elementary gradings and the above description of
the correspondence in Theorem \ref{A}.
\begin{corollary}\label{th:elementary}(see \cite[Proposition 4.11, Proposition 4.14]{cibils2010})
Let $\{x_1,\cdots,x_{{n}-1}\}$ be any set of generators of the free group $\mathcal{F}_{n-1}$,
and let $\tilde{x}=1+\sum_{i=1}^{n-1}x_i\in \N\mathcal{F}_{n-1}$. Then
\begin{enumerate}
\item
The grading class $\tilde{x}(\C)=\tilde{x}(\C\{e\})$ is the unique maximal connected grading class of $M_n(\mathbb{C})$ which is elementary.
\item
Any connected elementary grading class of $M_n(\mathbb{C})$ is a quotient of the elementary grading class $\tilde{x}(\C)$.
\end{enumerate}
\end{corollary}

\subsection{Proof of Theorem C}\label{proofD}
We extend the graded-simple case to general connected group gradings of finite-dimensional semisimple complex algebras.

For that we go back to the free product grading \eqref{coprodgrade}, focusing on
the case where the gradings $\mathcal{G}_j$ of $A_j$ are simple. For every $1\leq j\leq r$, let $G_j$ be a finite group and let
$$f_j\in Z^2(G_j,\C^*)$$ be any 2-cocycle. Let
$$\Gamma_j=\mathcal{F}_{m_j-1}*G_j$$
be the free product of $G_j$
and the free group $$\mathcal{F}_{m_j-1}=\langle x^{(j)}_{1},\cdots,x^{(j)}_{{m_j}-1}\rangle$$ of rank $m_j-1$.
For every $1\leq j\leq r$, let
$$\tilde{x}^{(j)}=1+\sum_{i=1}^{m_j-1}x^{(j)}_i\in\N \Gamma_j.$$
Lemma \ref{xcfg} tells us that for every $1\leq j\leq r$, the connected simple grading classes
$\tilde{x}^{(j)}(\C^{f_j}G_j)$ of the semisimple algebras $$A_j=\bigoplus_{[W]\in\text{Irr}(G_j,f_j)}M_{n(W)}(\C),\ \ n(W)=m_j\cdot\dim_{\C}(W)$$
are maximal.
These grading classes give rise to a connected grading class
\begin{equation}\label{coprod}
\coprod_{j=1}^r\tilde{x}^{(j)}(\C^{f_j}G_j)
\end{equation}
of the semisimple algebra $A=\bigoplus_{j=1}^rA_j$
by the free product $\coprod_{j=1}^r\Gamma_j$ as described in \eqref{coprodgrade}.
\begin{lemma}\label{scgr2}
Any connected grading class of a finite-dimensional complex semisimple algebra $A$ is a quotient of a grading class of the form \eqref{coprod}.
\end{lemma}
\begin{proof}
Let $[\mathcal{H}]$ be a connected grading class of $A$ by a group $\Gamma$.
By Theorem \ref{prop:decsemsi}, this grading class admits a decomposition $A=\bigoplus_{j=1}^rA_j$ to graded-simple ideals $A_j$, each of which endowed
with a connected simple grading class (by an appropriate subgroup) $[\mathcal{H}]_j$.
Theorem \ref{th:BSZ} now says that for every $1\leq j\leq r$ there exist $y_j\in \N \Gamma$ and finite groups $H_j<\Gamma$ such that $[\mathcal{H}]_j$ is the grading class
$y_j(\C^{\tilde{f}_j}H_j)$ for some $\tilde{f}_j\in Z^2(H_j,\C^*)$. Let $m_j=\epsilon(y_j)$.
Using once again the normalization as described in Remark \ref{1inx}, we can assume that $y_j=1+z_j$ for some $z_j=\sum_{i=1}^{m_j-1}g^{(j)}_i\in \N\Gamma$.
For every $1\leq j\leq r$ let $G_j$ be finite groups isomorphic to $H_j$ via $\phi_j: G_j\to H_j$.
Then the result follows by applying the surjective group homomorphism
$$\coprod_{j=1}^r\mathcal{F}_{m_j-1}*G_j\to \Gamma,$$
which extends the isomorphisms $\phi_j: G_j\to H_j$ by choosing generators $x^{(j)}_1,\cdots,x^{(j)}_{{m}-1}$ of $\mathcal{F}_{m_j-1}$
for every $1\leq j\leq l$ and sending them to $g^{(j)}_1,\cdots,g^{(j)}_{{m}-1}$ respectively.
\end{proof}
By Lemma \ref{scgr2}, the maximal connected grading classes of the semisimple algebras are of the form \eqref{coprod}.
The following example shows that not every such grading class is maximal.
\begin{example}\label{diagquot}
The ordinary group algebra $\C G$ of an abelian group $G$ of order $n>1$ is isomorphic to the commutative algebra $A=\oplus_{j=1}^n\C$.
Let $1(\C\{e\})$ be the induction of the trivial group algebra by a 1-dimensional space.
Then with the above notation, the class $\coprod_{j=1}^n1(\C\{e\})$ is grading $A$ by the trivial group $\coprod_{j=1}^n\{e\}=\{e\}$. This class is of the form \eqref{coprod},
but is a quotient of the $G$-grading class $1(\C G)$ under the forgetful morphism (see \S\ref{qg}).
More generally, if a grading class \eqref{coprod} admits $n$ factors of the form $1(\C\{e\})$, then
this grading class can be written as $$\coprod_{j=1}^{n}1(\C\{e\})*\coprod_{j=n+1}^r\tilde{x}^{(j)}(\C^{f_j}G_j).$$
This class is a quotient grading of
$$\C G*\coprod_{j=n+1}^r\tilde{x}^{(j)}(\C^{f_j}G_j),$$
for some abelian group $G$ of order $n$. Clearly, if $n>1$, this quotient is proper and thus the grading class \eqref{coprod} is not maximal.
\end{example}
In order to grasp the maximal grading classes of our algebras, consider quotients of the maximal graded-simple classes $\tilde{x}(\C^fG)$.
By Theorem \ref{prop:decsemsi} and Theorem \ref{th:BSZ}, such quotients are represented by
\begin{equation}
\bigoplus_{j=1}^ry_j(\C^{\hat{f}_j}H_j).
\end{equation}
In the following lemma we assume that these quotients are of the form \eqref{coprod}.

\begin{lemma}\label{simpletoprod}
Let
\begin{equation}\label{quotientequation}
[(\psi,\phi)]:\tilde{x}(\C^fG)\to\coprod_{j=1}^r\tilde{x}^{(j)}(\C^{f_j}G_j)
\end{equation}
be a quotient morphism class between a simple connected $\mathcal{F}_{m-1}*G$-grading class of the form \eqref{scgr} and a connected
$\coprod_{j=1}^r\mathcal{F}_{m_j-1}*G_j$-grading class of the form \eqref{coprod}.
If $r>1$, then $\tilde{x}(\C^fG)$ is equivalent to a twisted group algebra $\C^fG$, where $G$ is abelian and $[f]\in H^2(G,\C^*)$ is trivial.
Moreover, $[(\psi,\phi)]$ is the forgetful morphism class.
\end{lemma}
\begin{proof}
By Lemma \ref{lemma:quoinduce},
\begin{equation}\label{copquot}
\coprod_{j=1}^r\tilde{x}^{(j)}(\C^{f_j}G_j)=\bar{\phi}(\tilde{x})\cdot\coprod_{j=1}^ry^{(j)}(\C^{f_j}G_j),
\end{equation}
where $\bar{\phi}$ is the extension (see \eqref{extended}) of a group morphism representative $\phi$ and
$\coprod_{j=1}^ry^{(j)}(\C^{f_j}G_j)$ is the corresponding quotient grading class of the graded subalgebra $\C^fG$.
For every $j=1,\cdots, r$ we denote $\epsilon_j=\epsilon(y^{(j)})$ and normalize $$y^{(j)}=1+\sum_{i=1}^{\epsilon_j-1}y^{(j)}_i$$ according to Remark \ref{1inx}.
Then by \eqref{copquot}, for every $j$, the grading class $\tilde{x}^{(j)}(\C^{f_j}G_j)$ is the same as the class
$$\bar{\phi}(\tilde{x})\cdot y^{(j)}(\C^{f_j}G_j)=[\bar{\phi}(\tilde{x})+\sum_{i=1}^{\epsilon_j-1}y^{(j)}_i\cdot\bar{\phi}(\tilde{x})](\C^{f_j}G_j).$$
By Lemma \ref{hatvstilde}, for every $j=1,\cdots, r$
$$(\tilde{x}^{(j)},\bar{\phi}(\tilde{x})+\sum_{i=1}^{\epsilon_j-1}y^{(j)}_i\cdot\bar{\phi}(\tilde{x}))\in R_{G_j}.$$
So, for every $j=1,\cdots, r$ there exist partial sums $\tilde{z}^{(j)}_1,\tilde{z}^{(j)}_2\in \N[\mathcal{F}_{m_j-1}*G_j]$
such that $\tilde{z}^{(j)}_1+\tilde{z}^{(j)}_2=\tilde{x}^{(j)}$ and
\begin{equation}\label{ztilde}
(\tilde{z}^{(j)}_1,\bar{\phi}(\tilde{x}))\in R_{G_j}.
\end{equation}
From \eqref{ztilde} we get
\begin{equation}\label{cap}
\bar{\phi}(\tilde{x})\in \bigcap_{j=1}^r\N[\mathcal{F}_{m_j-1}*G_j]\subset\N[\coprod_{j=1}^r\mathcal{F}_{m_j-1}*G_j].
\end{equation}
Since the intersection $$\bigcap_{j=1}^r\mathcal{F}_{m_j-1}*G_j<\coprod_{j=1}^r\mathcal{F}_{m_j-1}*G_j$$ is trivial, then \eqref{cap} can hold only if $\bar{\phi}(\tilde{x})$ is in $\N$.
In other words $\epsilon(\bar{\phi}(\tilde{x}))=1,$ yielding
\begin{equation}\label{epsilontilde}
\epsilon(\tilde{x})=1.\end{equation}
Condition \eqref{epsilontilde} says that $\tilde{x}(\C^fG)$ is graded by the group $G$ which is finite.
Thus, \eqref{quotientequation} is a quotient morphism of algebras graded by finite groups.
Now, a free product of groups can be finite only if all, perhaps except one, of them are trivial. Suppose then that $G_1,\cdots,G_{r-1}$ are trivial and so the quotient is graded by
$G_r$, which might as well be trivial.
Comparing dimensions we get
\begin{equation}\label{GrGr}
|G|=r-1+|G_r|.
\end{equation}
Let $N\lhd G$ be the kernel of the group morphism $\phi$. Then on one hand, the order of $N$ is equal to the dimension of the trivial component of the quotient $G/N$-grading, that is
\begin{equation}\label{Nr}
|N|=r.
\end{equation}
On the other hand, by connectivity,
\begin{equation}\label{GrG/N}
G_r\cong G/N.
\end{equation}
Combining \eqref{GrGr}, \eqref{Nr} and \eqref{GrG/N}, we get
$$r-1+|G_r|=|G|=|N|\cdot|G/N|=r\cdot|G_r|.$$
Hence, if $r>1$ then $|G_r|=1$, and so the quotient is graded by the trivial group. This means that any representative $(\psi,\phi)$ is the forgetful morphism of the twisted group algebra $\C^fG$.
Moreover, since the simply-graded components of the quotient are all 1-dimensional, the underlying algebra is commutative, saying that $G$ is abelian and,
by Lemma \ref{lemma:cohdiv}, $f$ is cohomologically trivial.
\end{proof}

Lemma \ref{simpletoprod} is complementary to Lemma \ref{xcfg}. These two claims say a quotient morphism \eqref{quotientequation}
is either a graded-equivalence for $r=1$, or a forgetful morphism of an abelian group algebra for $r>1$. In order to move to the general case, that is a quotient morphism between
two grading classes of the form \eqref{coprod}, we say that a factor $\tilde{x}^{(j)}(\C^{f_j}G_j)$ of \eqref{coprod} is {\it trivial} if it grades a trivial summand $\C$, that is
if $$\epsilon(\tilde{x}^{(j)})\cdot|G_j|=1.$$
\begin{corollary}\label{prod}
A grading class \eqref{coprod} is maximal if and only if it contains no more than one factor $\tilde{x}^{(j)}(\C^{f_j}G_j)$ which is trivial.
\end{corollary}
\begin{proof}
Example \ref{diagquot} shows that if the grading class \eqref{coprod} contains more than one trivial factor $\tilde{x}^{(j)}(\C^{f_j}G_j)=1(\C\{e\})$ (graded by the trivial group $\{e\}$),
then it is a nontrivial
quotient of a class which admits a graded-simple factor $1(\C G)$, the ordinary group algebra of a nontrivial abelian group $G$. It is thus not maximal, proving the ``only if" part.
Conversely, suppose that the grading class $\coprod_{j=1}^r\tilde{x}^{(j)}(\C^{f_j}G_j)$ admits no more than one trivial summand. Assume also that it is a quotient of another grading class.
By Lemma \ref{scgr2}, we may assume that it is a quotient of another grading class of the form \eqref{coprod}
\begin{equation}\label{another}
[(\hat{\psi},\hat{\phi})]:\coprod_{k=1}^s\hat{x}^{(k)}(\C^{\hat{f}_k}\hat{G}_k)\to\coprod_{j=1}^r\tilde{x}^{(j)}(\C^{f_j}G_j).
\end{equation}
Both gradings determine decompositions of the underlying semisimple algebras to their simply graded parts. Any algebra isomorphism representative
$\hat{\psi}$ yields a partition $$\{1,\cdots,r\}=\bigcup_{k=1}^s\psi'(k)$$
and partial quotient morphism classes
$$[(\hat{\psi}_k,\hat{\phi}_k)]:\hat{x}^{(k)}(\C^{\hat{f}_k}\hat{G}_k)\to\coprod_{j\in\psi'(k)}\tilde{x}^{(j)}(\C^{f_j}G_j)$$
for every $1\leq k\leq s$.
Fix $k\in\{1,\cdots,s\}$.
If $|\psi'(k)|>1$, then by Lemma \ref{simpletoprod} $[(\hat{\psi}_k,\hat{\phi}_k)]$ is the forgetful morphism class of a twisted group algebra $\C^{\hat{f}_k}\hat{G}_k$ for
a nontrivial abelian group $\hat{G}_k$ and $\hat{f}_k$ cohomologically trivial.
In particular, $\coprod_{j\in\psi'(k)}\tilde{x}^{(j)}(\C^{f_j}G_j)=\coprod_{j\in\psi'(k)} 1(\C\{e_j\})$ is a product of $(1<)|\psi'(k)|$-many trivial factors.
But this is impossible, because the hypothesis does not allow more than one trivial factor in the grading class $\coprod_{j\in\psi'(k)}\tilde{x}^{(j)}(\C^{f_j}G_j)$.
We deduce that $|\psi'(k)|=1$ and so by Lemma \ref{diagquot}, $(\hat{\psi}_k,\hat{\phi}_k)$ is a graded-equivalence for every $k=1,\cdots,s$.
We get that the entire quotient morphism class \eqref{another} is the identity. By the definition of the quotient-grading partial order, the grading class \eqref{coprod} is maximal.
\end{proof}
We are now ready to prove Theorem C.
As can be expected, the formal term $\gamma^m$ is assigned to the maximal connected simple grading class \eqref{scgr}, where $\gamma$ is the Aut$(G)$-orbit of $[f]\in H^2(G,\C^*)$.
This can naturally be extended to a one-to-one correspondence
$$\coprod_{j=1}^r(1+\sum_{i=1}^{m_j-1}x^{(j)}_i)(\C^{f_j}G_j)\longleftrightarrow\sum_{j=1}^r\gamma_j^{m_j}$$
between the connected grading classes \eqref{coprod} of a finite-dimensional semisimple algebra $A$, and the formal sums $\sum a{(\gamma,m)}\gamma^{m}$ such that
$$\chi_A(z)=\sum a{(\gamma,m)}m^{-z}\zeta_{\gamma}(z).$$
By Theorem \ref{th:AHequi} and Lemma \ref{decompequiv}, two grading classes of the form \eqref{coprod} are the same
if and only if they are mapped to the same formal sum.
Next, by Lemma \ref{scgr2}, any maximal connected grading class of $A$ is of the form \eqref{coprod}, while Corollary \ref{prod} tells us that
grading class are maximal if and only if they contain at most one non-trivial factor, which in turn says that the corresponding formal sum $\sum_{j=1}^r\gamma_j^{m_j}$ admits no more than
one non-trivial terms. This completes the proof.
\qed

\subsection{Diagonal algebras}\label{diagsec}
In his paper \cite{das2008}, S. D\u{a}sc\u{a}lescu studies the group gradings of the so called {\it diagonal algebra}, that is the direct sum of a certain field.
He tacitly classifies those of them which are maximal connected. Let us describe these maximal classes over the field of complex numbers using the above notation.
Obviously, an induction of any graded algebra by the endomorphisms of a graded space of dimension higher than $d$ admits graded ideals of dimension higher than $d^2$.
Hence, by Theorem \ref{th:BSZ} every graded-simple class of a diagonal algebra $\C^n$ is a twisted group algebra graded-class. Furthermore, if a twisted group algebra $\C^fG$ is commutative
then $G$ is obviously abelian and, by Lemma \ref{lemma:cohdiv}, $f$ is cohomologically trivial.
By the discussion in \S\ref{proofD} we obtain
\begin{theorem}\label{da}(see \cite[Theorem 5]{das2008})
With the above notation, the maximal connected grading classes of the diagonal algebra $\C^n$ are free products $\coprod_{j=1}^r\C G_j$ of ordinary group algebras,
where $G_j$ are abelian groups with $\sum_{j=1}^r|G_j|=n,$ such that no more than one group $G_j$ is trivial.
\end{theorem}
Note that Theorem \ref{da} generalizes \cite[Proposition 6.1, Corollary 6.2]{cibils2010}.
In particular, it shows that for every $n\geq 3$ ($n$ can be square-free) there is no universal covering of $\C^n$.

\subsection{Bisimple algebras}\label{bisimpsec}
A finite-dimensional complex algebra is {\it bisimple} if it is of the form
\begin{equation}\label{bisimp}
A=M_{n_1}(\C)\oplus M_{n_2}(\C).
\end{equation}
Our investigation of maximal connected grading classes of a
bisimple algebra \eqref{bisimp} starts with those classes which
are graded-simple. Any graded-simple class of a bisimple algebra
is of the form $x(\C^fG)$, where the twisted group algebra $\C^fG$
is itself a bisimple algebra, that is $|$Irr$(G,f)|=2$. A group
$G$ satisfying this property is termed {\it of $2f$-central type}
in \cite{Higgs88}, or just {\it of double central type}. For
example, the cyclic group $C_2$ is $2f$-central simple for any
$f\in Z^2(C_2,\C^*)$ (they are all cohomologous). In the spirit of
Corollary \ref{regularity}, apply Theorem \ref{rayclasses}
to obtain the following alternative definition.
\begin{lemma}\label{biregularity}
A cocycle $f\in Z^2(G,\C^*)$ satisfies $|$Irr$(G,f)|=2$ if and only if there is a unique nontrivial $g\in G$, up to conjugation, such that res$|^G_{C_G(g)}\alpha_f(g,-)=1$.
\end{lemma}
Next, a class which is not graded-simple clearly decomposes to a direct sum of two graded-simple algebras, both are matrix algebras. Applying the discussion in \S\ref{proofD}
we get
\begin{theorem}\label{121bisimp}
There is a one-to-one correspondence between the maximal connected grading classes of a non-diagonal bisimple algebra \eqref{bisimp} and
\begin{enumerate}
\item  the pairs $(G,\gamma)$, where $\gamma$ are Aut$(G)$-orbits of
classes $[f]\in H^2(G,\mathbb{C}^*)$ for groups $G$ of $2f$-central type whose orders divide $n_1^2+n_2^2$, as well as
\item the quadruples of the form
$(G_1,\gamma_1,G_2,\gamma_2)$, where for $i=1,2$, $G_i$ are groups of central type and of orders dividing $n_i^2$, and $\gamma_i$ are Aut$(G_i)$-orbits of
non-degenerate cohomology classes in $H^2(G_i,\mathbb{C}^*)$. If additionally $n_1=n_2$, then the quadruples $(G_1,\gamma_1,G_2,\gamma_2)$ and $(G_2,\gamma_2,G_1,\gamma_1)$ are identified.
\end{enumerate}
\end{theorem}
The diagonal algebra $\C\oplus \C$ is not covered by Theorem \ref{121bisimp}. Here, by Theorem \ref{da}, there is a unique maximal connected grading class, namely by the $2f$-central type $C_2$.
\begin{remark}\label{remarkbisimp}
It is proven \cite[Theorem A]{Higgs88} that if $G$ is $2f$-central type, then the two non-equivalent irreducible $f$-projective representations of $G$ are of the same dimension
(in particular $G$ is of order $2m^2$ for some positive integer $m$), i.e.
$$\C^fG=M_{\sqrt{\frac{|G|}{2}}}(\C)\oplus M_{\sqrt{\frac{|G|}{2}}}(\C).$$
Consequently, the first family of classes in Theorem \ref{121bisimp} is empty unless $n_1=n_2$.
In fact, as we shall soon see in Proposition \ref{CGDneces}, the equality $n_1=n_2$ is also a sufficient condition for a bisimple algebra \eqref{bisimp} to admit graded-simple classes.
\end{remark}

Let us study some special groups of double central type.
\begin{proposition}\label{index2}
Suppose that the restriction of a cocycle $f\in Z^2(G,\C^*)$ to a subgroup of $N<G$ of index 2 is non-degenerate.
Then $G$ is of $2f$-central type.
\end{proposition}
\begin{proof}
Let $V$ be an $f$-projective representation of $G$. Then $V$ decomposes over $N$ as a direct sum of copies of the unique irreducible $f|_N$-projective representation $W$,
that is $V|_N=W^r$ for some $r\geq 1$.
Lemma \ref{dimprojrep} says that dim$_{\C}(W)=\sqrt{|N|}$ and hence
$$|G|\geq\dim_{\C}(V)^2=r^2\text{dim}_{\C}(W)^2=r^2|N|=r^2\frac{|G|}{2}.$$
We get that $r=1$, and so $V|_N=W$ is irreducible over $N$, that is dim$_{\C}(V)=\sqrt{\frac{|G|}{2}}$. Therefore, $G$ admits precisely two irreducible $f$-representations.
This proves the claim.\end{proof}
In the following two theorems we use again the notation $f_d$ for the restriction of $f\in Z^2(G,\C^*)$ to a Hall subgroup $G_d<G$.
\begin{theorem}\label{sylowdoublect}
Let $G$ be of $2f$-central type. Then for every odd $d$, $f_d$ is non-degenerate.
\end{theorem}
\begin{proof}
The theorem is proven in a similar way to the proof of Theorem \ref{th:hall} given in \cite[Lemma 2.7]{Passman2010}.
Let $W_1$ and $W_2$ be the two non-equivalent irreducible $f$-projective representations of $G$.
By Remark \ref{remarkbisimp}, $W_1$ and $W_2$ are of the same dimension, namely $m:=\sqrt{\frac{|G|}{2}}$. This is also the multiplicity of $W_1$ and $W_2$ in $\C^fG$,
viewed as the regular module over itself. It can hence be written as
\begin{equation}\label{samemult}
\C^fG=W_1^{m}\oplus W_2^{m}.
\end{equation}

Next, consider $\C^fG$ as a free module over $\C^{f_d}G_d$ of rank $[G:G_d]$ (the index of $G_d$ in $G$).
Then any irreducible $f_d$-projective representation $W$ of $G_d$ is of multiplicity dim$_{\C}(W)\cdot[G:G_d]$ in the $\C^{f_d}G_d$-module $\C^fG$.
Any such irreducible $f_d$-projective representation appears as a constituent of the restriction to $G_d$ of either $W_1$ or $W_2$.
By \eqref{samemult}, the multiplicity of an irreducible $f_d$-projective representation in the $\C^{f_d}G_d$-module $\C^fG$ is divisible by $m$.
Consequently, $m$ divides $\dim_{\C}(W)\cdot[G:G_d].$
We compute the $d$-part of both sides, taking into account that the index $[G:G_d]$ is prime to $d$. The $d$-part $m_d$ divides $\dim_{\C}(W)$,
and hence $(\frac{|G|}{2})_d$ divides $\dim_{\C}(W)^2$. Since $d$ is odd we obtain that $|G|_d=|G_d|$ divides $\dim_{\C}(W)^2$.
This can happen when the order of $G_d$ is exactly equal to $\dim_{\C}(W)^2$,
equivalently, when the regular module $\C^{f_d}G_d$ is decomposed as $$\C^{f_d}G_d=W^{\dim_{\C}(W)}.$$
In other words, Irr$(G_d,f_d)=\{[W]\}$, and $f_d$ is non-degenerate.
\end{proof}
\begin{theorem}
Let $m$ be an odd integer. A double central type group of order $2m^2$ is a semidirect product $G=G_{2'}\rtimes_{\eta}C_2$,
where the $2'$-Hall subgroup $G_{2'}$ is of central type and the action $\eta$ is symplectic.
\end{theorem}
\begin{proof}
Let $G=G_{2'}\rtimes_{\eta}C_2$, where the action $\eta$ stabilizes some non-degenerate cohomology class $[f_{2'}]\in H^2(G_{2'},\C^*)$.
Then by \eqref{eq:semdircoh}, $[f_{2'}]$ is restricted from a class $[f]\in H^2(G,\C^*)$. From Proposition \ref{index2} we deduce that $G$ is of $2f$-central type.
Conversely, a $2f$-central type group $G$ is solvable \cite[Theorem B]{Higgs88}, therefore it has a $2'$-Hall subgroup $G_{2'}$. This subgroup is of index 2 in $G$, and hence it is normal.
By Theorem \ref{sz}, $G=G_{2'}\rtimes_{\eta} C_2$ for some action $\eta:C_2\to$Aut$(G_{2'})$.
By Theorem \ref{sylowdoublect}, the restriction $f_{2'}$ of $f$ to $G_{2'}$ is non-degenerate.
In particular, $G_{2'}$ is of central type and $f_{2'}$ is $\eta$-invariant, that is $\eta$ is symplectic.
\end{proof}

\subsection{Completely graded decomposition algebras}\label{cgda}
We say that a connected grading of a finite-dimensional semisimple complex algebras is {\it completely decomposed} if all its graded-simple summands are simple.
An algebra \eqref{eq:kraoto3} is {\it completely graded decomposition} (CGD) if all its connected gradings are completely decomposed.
A primal family of examples would be bisimple algebras \eqref{bisimp} with $n_1\neq n_2$. Theorem \ref{121bisimp} and Remark \ref{remarkbisimp} ensure that these algebras are CGD.
\begin{proposition}\label{CGDneces}
An algebra \eqref{eq:kraoto3} is CGD only if all the multiplicities $b(n)$ are at most 1.
\end{proposition}
\begin{proof}
Assume that there exists a multiplicity $b(n)$ which is larger than 1. Then the subalgebra $M_n(\C)^{b(n)}$ is not simple but admits the graded-simple class $n(\C G)$
for any abelian group $G$ of order $b(n)$. Extending this class to a connected $G$-grading class of $A$ by grading all the other components by, say, the trivial element, we deduce that $A$ is not CGD.
\end{proof}
The converse direction of Proposition \ref{CGDneces} is likely to hold.
\begin{conjecture}\label{conjCDG}
An algebra \eqref{eq:kraoto3} is CGD if all the multiplicities
$b(n)$ are at most 1.
\end{conjecture}
Conjecture \ref{conjCDG} is actually a different way to formulate a conjecture of R.J. Higgs \cite{higgs}, which says that a cocycle $f\in Z^2(G,\C^*)$ ($G$ is finite) is either non-degenerate, or
there exist distinct $[W_1],[W_2]\in$Irr$(G,f)$ of the same dimension.
Indeed, if $f\in Z^2(G,\C^*)$ refutes Higgs' conjecture, then the simple components of the (non-simple) algebra $A=\bigoplus_{[W]\in\text{Irr}(G,f)}\text{End}_{\C}(W)$ are of distinct dimensions.
Nevertheless, the connected grading class $\C^fG$ of $A$ is certainly not completely decomposed. Hence the semisimple algebra $A$ is not CGD, and a counterexample to Conjecture \ref{conjCDG} is found.
Conversely, suppose that all the multiplicities $b(n)$ of an algebra \eqref{eq:kraoto3} are 1. If $A$ is not CGD, and by that is a counterexample to Conjecture \ref{conjCDG},
then it admits a connected grading which is not completely decomposed.
Then one of its graded-simple summands is not simple. Then the grading on this summand is induced from a twisted group algebra which refutes Higgs' conjecture.

By the discussion in \S\ref{proofD} and Proposition \ref{CGDneces} we get
\begin{theorem}
There is a one-to-one correspondence between the maximal connected grading classes of a CGD algebra $A=\bigoplus_{i=1}^sM_{n_i}(\C)$ and
the $2s$-tuples of the form
$(G_1,\gamma_1,G_2,\gamma_2, \cdots, G_s,\gamma_s)$, where for $i=1,\cdots, s$, $G_i$ are groups of central type and of orders dividing $n_i^2$, and $\gamma_i$ are Aut$(G_i)$-orbits of
non-degenerate cohomology classes in $H^2(G_i,\mathbb{C}^*)$.
\end{theorem}

To provide evidence to Conjecture \ref{conjCDG}, we need the following folklore result.
\begin{lemma}\label{lemma:cohdiv} (see \cite[Lemma 1.2(i)]{Higgs88})
Let $G$ be a finite group. The order of any cohomology class $[f]\in H^2(G,\C ^*)$ divides the dimension of each $f$-projective representation.
In particular, if there is an $f$-projective representation of dimension 1, then $f$ is cohomologically trivial.
\end{lemma}
\begin{theorem}\label{pfactor<3}
Let $n_1,\cdots, n_s$ be distinct positive integers such that every prime $p$ is a factor of no more than two of them. Then $A=\bigoplus_{i=1}^sM_{n_i}(\C)$ is CGD.
\end{theorem}
\begin{proof}
Suppose, by negation, that there exists a connected grading of $A$ which is not completely decomposed. Then it admits a graded-simple ideal which is not simple.
By Theorem \ref{th:BSZ}, this graded ideal is induced from a twisted group algebra, say $\C^fG$, which decomposes to at least two simple algebras, all of distinct dimensions.
By Remark \ref{remarkbisimp}, this decomposition must admit strictly more than two simple summands.
Now, by Lemma \ref{lemma:cohdiv}, the order of $[f]$ in $H^2(G,\C^*)$ must divide the dimensions of all the
simple summands of $\C^fG$. However, there are more than two of those, and so by the hypothesis no prime number divides them all. Consequently, $f$ is cohomologically trivial. This is one of
the cases where Higgs' conjecture is known to hold. More precisely, by \cite[Lemma 1.1]{higgs} $\C^fG$ must admit two simple summands of the same dimension, which is a contradiction.
\end{proof}

Theorem D is a consequence of Theorem \ref{pfactor<3} and the following claim.
\begin{theorem}
Let $A=\bigoplus_{i=1}^sM_{n_i}(\C)$ be a CGD algebra. Then $$\pi_1(A)=\coprod_{i=1}^s\pi_1(M_{n_i}(\C)).$$
\end{theorem}
\begin{proof}
By induction on $s$, using Lemma \ref{+free} and the CGD definition.
\end{proof}

\end{document}